\documentclass[11pt]{article}
\usepackage{amsmath}
\usepackage{amssymb}
\usepackage{latexsym}
\usepackage{amsthm}
\usepackage{fullpage}
\usepackage[colorinlistoftodos]{todonotes}
\usepackage{marginnote}
\usepackage{enumitem}
\usepackage{mathtools}
\usepackage{dsfont}
\usepackage{mdframed}
\usepackage{xcolor}
\usepackage{hyperref}
\usepackage{esint}
\usepackage{amsfonts,latexsym,amsmath,amscd,geometry}
\numberwithin{equation}{section}

\newtheorem{theorem}{Theorem}[section]
\newtheorem{lemma}[theorem]{Lemma}

\newtheorem{remark}[theorem]{Remark}
\newtheorem{definition}[theorem]{Definition}

\definecolor{vine}{rgb}{0.7,0.1,0.1}
\definecolor{vine}{rgb}{0.7,0.1,0.1}
\definecolor{darkgreen}{rgb}{0,0.5,0}

\newcommand{\R}{{\mathbb R}}
\newcommand{\N}{{\mathbb N}}

\newcommand{\eps}{\varepsilon}
\renewcommand{\d}{{ \mathrm{d}}}

\newcommand{\no}[2]{ \left\| #1 \right\|_{#2} }

\renewcommand{\div}{\operatorname{div}\,}

\newcommand{\weak}{\rightharpoonup}

\newcommand{\Sp}{{\mathbb{S}}}

\newcommand{\loc}{{\operatorname{loc}}}

\newcommand{\Id}{{\mathbb{I}}}

\newcommand{\EE}{\textbf{e}}

\newcommand{\spc}{\,\,\,\,\,\,\,}
\newcommand{\ZZ}{{\mathcal Z}}
\newcommand \nc{\newcommand}

\nc{\ba}{\begin{array}}\nc{\ea}{\end{array}}
\nc{\be}{\begin{eqnarray}}\nc{\ee}{\end{eqnarray}}
\nc{\beq}{\begin{equation}}\nc{\eeq}{\end{equation}}
\nc{\bex}{\begin{eqnarray*}}\nc{\eex}{\end{eqnarray*}}
\nc{\btm}{\begin{theorem}} \nc{\etm}{\end{theorem}}
\nc{\blm}{\begin{lemma}} \nc{\elm}{\end{lemma}}
\nc{\ld}{\lambda}
\nc{\va}{\varphi}
\nc{\ve}{\varepsilon}

\def\pa{\partial}
\def\pf{\noindent{\bf Proof.\quad}}

\author{Joshua Kortum\footnote{Julius-Maximilians-Universität W\"urzburg, Institute of Mathematics, Emil-Fischer-Straße 40, 97074 Würzburg, Germany (joshua.kortum@uni-wuerzburg.de)}, \ \ Changyou Wang\footnote{Department of Mathematics, Purdue University, West Lafayette, Indiana, 150 N. University Street, West Lafayette, IN 47907--2067, (wang2482@purdue.edu)}}
\title{Existence and compactness of global weak solutions of three-dimensional axisymmetric Ericksen-Leslie system}

\date{}

\begin{document}
\maketitle

\begin{abstract}
In dimension three, the existence of global weak solutions to the axisymmetric simplified Ericksen-Leslie system without swirl is established.   This is achieved by analyzing weak convergence of solutions of the axisymmetric Ginzburg-Landau approximated solutions as the penalization parameter $\varepsilon$ tends to zero.  The proof relies on the one hand on the use of a blow-up argument to rule out energy concentration off the $z$-axis, which  exploits the topological restrictions of the axisymmetry. On the other hand, possible limiting energy concentrations on the $z$-axis can be dealt by a cancellation argument at the origin. Once more, the axisymmetry plays a substantial role. We will also show that
the set of axisymmetric solutions without swirl $(u,d)$ to the simplified Ericksen-Leslie system is compact under weak convergence in $L^\infty_tL^2_x\times L^2_tH^1_x$.
\end{abstract}

%\medskip

{\bf AMS-Classification (2020): 35Q35, 35D30, 76A15}  \\

\section{Introduction}
In this paper, we consider the simplified form of the Ericksen-Leslie system in dimension three. It represents a set of partial differential equations modeling the hydrodynamics of nematic liquid crystals, which was proposed by Ericksen \cite{ericksen1962} and Leslie \cite{leslie1968} around the 1960's. In a simplified version proposed by Lin \cite{lin1989}, it  reads
\begin{align}
\partial_t u + (u \cdot \nabla ) u  - \mu \Delta u+ \nabla P &= -\lambda \div \left( \nabla d \odot \nabla d - \frac{|\nabla d|^2}{2} \Id_3 \right) ,  \label{momentum} \\
																\div u & =0,\label{incompress}\\
				\partial_t  d + (u\cdot \nabla ) d  &= \gamma (  \Delta d + |\nabla d|^2 d). \label{director}
\end{align}
This system is posed on a bounded smooth domain $\Omega \subset \R^3$ and for an arbitrary time horizon $0<T\leq +\infty$. The constitutive quantities of the system are given by the velocity field of the nematic fluid  $u: \Omega\to \R^3$, a pressure function $P :\Omega \to \R$ and a unit vector field $d : \Omega \to \Sp^2$ displaying the macroscopic (average) orientation field of the nematic liquid crystal molecules. The positive, physical constants $\mu, \lambda$ and $\gamma$ represent the viscosity coefficient of the fluid, the competition of potential and kinetic  energy of the fluid and the microscopic elastic relaxation time for the molecular orientation field, respectively. Furthermore, the matrix $\mathbb{I}_3$ is the three-dimensional identity matrix and $\nabla d \odot \nabla d= [\partial_i d \cdot \partial_j d]_{i,j=1,2,3}$. Being a dissipative type equation, the system \eqref{momentum}--\eqref{director} is supplemented by initial and boundary conditions 
\begin{align}
	& (u,d)(x,0) 	 =  (u_0,d_0)(x), 	& x \in \Omega ,  \label{iniData}\\
	& (u,d)(x,t)   = (0,d_0(x,t) )   										& x \in \partial \Omega, ~t >0 \label{boundaryData}
\end{align}
for given initial data $(u_0,d_0) : \Omega \to \R^3 \times \Sp^2$ with $\div u_0 =0$ and $|d_0|\equiv 1$ on $\Omega$. In particular, the non-slip boundary condition for $u$ and the Dirichlet boundary condition for $d$ is assumed throughout this paper.

System \eqref{momentum}--\eqref{director} in principle consists  of the Navier-Stokes equations coupled with the heat flow of harmonic maps into the unit sphere. The strong coupling, in particular the Ericksen stress tensor induced by $d$ on the right-hand side of the momentum equation \eqref{momentum}, significantly complicates the establishment of an existence/well-posedness theory to the system. 
There has seen a substantial amount of studies devoted to the system \eqref{momentum}--\eqref{director} over the past several years. 
In dimension two, the problem has been essentially solved. For instance, partly inspired by the seminal work of Struwe \cite{struwe1985} on the heat flow of harmonic maps from Riemann surfaces, Lin-Lin-Wang \cite{lin2010} have established the global existence of Leray-Hopf type weak solutions of \eqref{momentum}--\eqref{director}  over any bounded domain $\Omega\subset\mathbb R^2$ with at most finitely many singular times (see Hong \cite{hong2010} for a similar result when $\Omega=\mathbb R^2$). And extensions of \eqref{momentum}-\eqref{director} to general Ericksen-Leslie systems have also been studied by Huang-Lin-Wang \cite{huang2014}, and the uniqueness of solutions was established by Lin-Wang in \cite{lin2010b}. We would like to point out that in general global-in-time existence of smooth solutions of \eqref{momentum}--\eqref{director} cannot be expected. Similar to the heat flow of harmonic maps (see Chang-Ding-Ye \cite{chang1992}), examples of finite time singularities have been constructed for the simplified Ericksen-Leslie system \eqref{momentum}--\eqref{director} by Lai-Lin-Wang-Wei-Zhou \cite{lai2022} in dimension two and by Huang-Lin-Liu-Wang \cite{huang2016} in dimension three. 
When the dimension $n=3$, there are very few results available in the literature: 
For any initial data $(u_0, d_0)$, with small norm $\|u_0\|_{{\rm{BMO}}^{-1}(\mathbb R^n)}+\|d_0\|_{\rm{BMO}(\mathbb R^n)}$
for $n\ge 2$, Wang \cite{wang2011} has established a well-posedness theory of \eqref{momentum}--\eqref{director}. For an initial data $(u_0, d_0)\in L^2_{\rm{div}}(\Omega)\times H^1(\Omega,\mathbb S^2)$ with $d_0(\Omega)\subset\mathbb S^2_+$ (the upper hemisphere) for any bounded domain $\Omega\subset\mathbb R^3$, a global Leray-Hopf type weak solution of \eqref{momentum}--\eqref{director} was shown to exist by Lin-Wang \cite{lin2016}. How to construct global weak solutions of \eqref{momentum}--\eqref{director} for initial data with finite energy has remained to be an open problem. We would like to mention that 
there has been local well-posedness and global stability theorem on general Ericksen-Leslie systems in certain Sobolev or Besov spaces by Hieber-Pr\"uss \cite{hieber2017, hieber2018} in the isothermal case, and by De Anna-Liu \cite{deanna2019} in the non-isothermal case. We also mention
the paper \cite{feng2020} by Feng-Hong-Mei  on the existence of local strong solutions of the Ericksen-Leslie system for initial data
in higher order Sobolev spaces.

In this paper, we will study the problem of existence of solutions of \eqref{momentum}--\eqref{director} in the axisymmetric setting through analyzing the convergence of
axisymmetric solutions to the corresponding Ginzburg-Landau approximation system, as the parameter $\varepsilon\to 0$. Recall that Lin and Liu \cite{lin1995} initiated the study of Ginzburg-Landau approximation of \eqref{momentum}--\eqref{director},
in which they have relaxed the nonlinear constraint $|d|=1$ by a penalization energy term reading in full
\begin{align*}
        E_\eps (d_\eps) = \int_\Omega \frac{|\nabla d_\eps|^2}{2} + \frac{(1-|d_\eps|^2)^2}{4 \eps^2}
\end{align*}
for $\eps>0$ and $d_\eps:\Omega \to \R^3$. Consequently, equation \eqref{director} is replaced by 
\begin{equation} \label{director1}
\pa_t d_\eps + (u_\eps \cdot \nabla) d_\eps  = \gamma \left( \Delta d_\eps +\frac{1 - |d_\eps|^2}{\eps^2} d_\eps \right)
\end{equation}
and the well-posedness theory of \eqref{momentum},\eqref{incompress}, and \eqref{director1} is essentially limited by the well-posedness theory for the Navier-Stokes equation. In fact, Lin and Liu proved in \cite{lin1995} the existence of global-in-time smooth solutions in two dimensions, and global-in-time weak solutions as well as small strong solutions in three dimensions. A natural idea is then to send $\varepsilon \to 0^+$ and recover the original system. In three dimensions, this idea has been successfully implemented by \cite{lin2016} when the image of initial director $d_0$ is contained in $\Sp^2_+$. In two dimensions, for an arbitrary initial data of finite energy, the convergence as $\eps \to 0^+$ was proven by \cite{kortum2020} and \cite{du2022}. 
However, besides the mentioned cases this idea has not yielded a satisfactory existence theory of weak solutions in three dimensions. Our goal in this paper is to extend the theory to the three-dimensional axisymmetric case. 

Now let us recall the axisymmetric form of \eqref{momentum}--\eqref{director} without swirl, which was considered in \cite{huang2016}. 
If we denote the cylinderical coordinate frame by
\begin{align*}
    \EE_r= \begin{pmatrix}
        \cos \theta \\ \sin \theta \\0
    \end{pmatrix}, \quad 
    \EE_\theta= \begin{pmatrix}
        -\sin\theta \\ \cos \theta \\0
    \end{pmatrix}, \quad
    \EE_z= \begin{pmatrix}
        0 \\ 0 \\1
    \end{pmatrix},
\end{align*}
then we may write $u=u^r \EE_r +u^\theta \EE_\theta + u^z \EE_z$ and $d= d^r \EE_r + d^\theta \EE_\theta+ d^z \EE_z$.
Axisymmetry requires invariance under rotations around the $z$-axis of the quantities, i.e.\
$$ Q^\top \circ (u,d) \circ Q = (u,d)$$
for all $Q\in SO(3)$ being a rotation w.r.t.\ the $z$-axis. Furthermore, if $u^\theta \equiv 0, ~ d^\theta \equiv 0$,
we say $(u,d)$ is axisymmetric without swirl. In the axisymmetric case without swirl, the velocity, pressure and director field reduce to
\begin{align*}
        & u(r,\theta, z, t) = u^r(r,z,t) \EE_r + u^z(r,z,t) \EE_z, \\
        & d(r,\theta, z, t) = \sin \varphi(r,z,t) \EE_r + \cos \varphi(r,z,t) \EE_z, \\
        &P(r, \theta, z, t) = P(r,z,t)
\end{align*}
for a function $\varphi: \Omega \to \R$. In the following, we let  $\mu=\lambda = \gamma=1$ as it is not essential for the analysis. Then the axisymmetric version without swirl of \eqref{momentum}--\eqref{director}  reads
\begin{align}
	& \left( \pa_t + b \cdot \nabla  -  L  + \frac{1}{r^2} \right) u^r+P_r = -\left(L\varphi - \frac{\sin 2\varphi}{2r^2} \right) \varphi_r,
	\label{momentumRadial}\\
	& \left( \pa_t  +b\cdot \nabla   -  L \right) u^z +P_z = -\left(L\varphi - \frac{\sin 2\varphi}{2r^2} \right) \varphi_z, \label{momentumZ}\\
	&\frac1r (ru^r)_r +(u^z)_z=0, \\
	& \left( \pa_t   + b\cdot \nabla - L\right) \varphi = - \frac{\sin 2\varphi}{2r^2}, \label{HMF}
\end{align}
where 
\begin{align*}
	& b= u^r \EE_r + u^z \EE_z, \\
	& L:= \partial_r^2 + \frac{1}{r}\partial_r + \partial_z^2
\end{align*}
on an axisymmetric domain $\Omega\subset\mathbb R^3$, see Tsai \cite{tsai2018} for relevant discussion on the Navier-Stokes equation. In \cite{huang2016}, it was proven that the  axisymmetric solution without swirl in special cases blows up or remains smooth depending on topology of $d_0$. These results are in accordance with \cite{chang1992} on the heat flow of harmonic maps from surfaces, and the questions arise whether the blow-up solutions might be continued as weak solutions.

Throughout the paper, we write $\Sp^2$ for the standard $2$-sphere, $\Omega\subset \R^3$ is a contractible, axisymmetric and smooth domain, $T>0$ an arbitrary but fixed positive life span and $\ZZ = \{(x_1,x_2,x_3)\in \Omega | ~ x_3=0\}$ the $z$-axis. We use the standard notation for continuously differentiable, Sobolev \ spaces by $X= C^k(\Omega) , W^{k,p}(\Omega)$
for $k\ge 0$ and $1<p<\infty$, most of the time suppressing the target spaces. For (weakly) divergence-free function spaces, we write $X_{\div}$. 
Also write $H^1$ (or $H^1_0$) for $W^{1,2}$ (or $W^{1,2}_0$).

We first give the definition of a weak solution of 
\eqref{momentum}--\eqref{director}:
\begin{definition}		\label{defWeakSolution}
 A pair of functions
\begin{align*}
	u \in &  L^\infty (0,\infty; L^2_{\div}(\Omega) ) ~  \cap ~ L^2(0,\infty; H_{0}^1(\Omega)),  \\
  	d \in  &  L^\infty (0,\infty; H^1(\Omega,\Sp^2)) ~ \cap  ~H^1(0,\infty;L^{p}(\Omega))
\end{align*}
for some $p>1$ is called a weak solution
to the initial value problem \eqref{momentum}--\eqref{director} subject to the initial and boundary conditions \eqref{iniData}--\eqref{boundaryData} if 
\begin{align}
& \int_0^\infty \int_\Omega \left[ - u  \cdot\partial_t \phi - u \otimes u : \nabla \phi + \nabla u : \nabla \phi - \nabla d \odot \nabla d : \nabla \phi \right] \d x \d t  = \int_\Omega u_0 \cdot  \phi \d x , \label{weakMomentum} \\
& \int_0^\infty \int_\Omega \left[ \partial_t d \cdot \xi + (u \cdot \nabla ) d  \cdot \xi + \nabla d : \nabla
\xi - |\nabla d|^2 d \cdot \xi \right] \d x \d t  = 0 \label{weakDirector}
\end{align}
hold true for all $\phi \in C_{0,\div}^\infty( \Omega \times [0,\infty))$ and $ \xi \in C^\infty_0( \Omega \times [0,\infty))$. Additionally, $(u,d)$  attend the initial data $(u_0,d_0) \in L^2_{\div}(\Omega) \times H^1(\Omega, \Sp^2)$  in the weak sense, i.e.
\begin{align*}
	\int_\Omega u(t) \cdot \psi  \d x \to \int_\Omega u_0\cdot  \psi \d x, \qquad 
	\int_\Omega d(t) \cdot \zeta  \d x \to \int_\Omega d_0 \cdot\zeta  \d x
\end{align*}
for all $\psi \in C_{\div}^\infty( \Omega)$ and  $ \zeta \in C^\infty_0( \Omega)$ as $t \to 0^+$. Furthermore, the weak energy inequality is satisfied, i.e.\ 
\begin{align} \label{energyInequality}
\int_\Omega \left[ \frac{|u(t)|^2}{2}+ \frac{|\nabla d(t)|^2}{2} \right]\,dx 
&+ \int_0^t \int_\Omega \left[ |\nabla u|^2 + \left| \Delta d + |\nabla d|^2 d \right| \right]\,dxdt\nonumber\\
&\leq \int_\Omega \left[ \frac{|u_0|^2}{2}+ \frac{|\nabla d_0|^2}{2} \right]\,dx 
\end{align}
for a.e.\ $t>0$.
\end{definition}
The definition coincides with the usual definition of weak solutions for the heat flow of harmonic maps (c.f.\ \cite{lin2008}
) and the Navier-Stokes equations \cite{robinson2016}. 

Now we state our main result.

% \begin{theorem}  \label{thm:weakStability}
% Suppose, there is a sequence of axisymmetric initial data $(u^n(0), d^n(0))_n$ strongly converging  to $(u(0), d(0) )$ in $L^2_{x,\div}(\Omega) \times W^{1,2}_x(\Omega)$. Furthermore, suppose that the corresponding solutions $(u^n, P^n , d^n)$ to \eqref{momentumRadial}--\eqref{HMF} are smooth for every $n \in \N$. Then there exists a subsequence (not relabeled) and limits $(u,d)$ in the energy space such that 
% \begin{enumerate}[label=(\roman*)]
% 	\item $(u^n,d^n) \weak^* (u,d)$ in the energy space,
% 	\item $(u,d)$ satisfies the weak formulation \eqref{weakMomentum}--\eqref{weakDirector},
% 	\item $(u,d)$ is smooth outside $\{r=0\} \times (0,\infty)$ and
% 	\item the energy inequality \eqref{energyInequality} holds for $(u,d)$.
% \end{enumerate}
% \end{theorem}

\begin{theorem} \label{thm:weakExistence}
Suppose  that $\Omega$ is a simply-connected, axisymmetric smooth domain and $u_0\in L^2_{\div}(\Omega)$  and $d_0\in H^1(\Omega, \Sp^2)$, with $d_0\in H^{\frac32}(\partial\Omega,\mathbb S^2)$, are axisymmetric without swirl.  Then there exists 
a global weak solution $(u, d, P):\Omega\times [0,\infty)\to \mathbb R^3\times\mathbb S^2\times \mathbb R$ to \eqref{momentum}--\eqref{boundaryData} 
that is axisymmetric without swirl.
%in the sense of Definition \ref{defWeakSolution}. %Furthermore, the solution is smooth away from the $z$-axis and satisfies \eqref{momentumRadial}--\eqref{HMF} pointwise.
\end{theorem}

As a byproduct of the arguments to prove Theorem \ref{thm:weakExistence},
we will establish a weak compactness property of the axisymmetric with no swirl solutions of \eqref{momentum}--\eqref{boundaryData}. More precisely, we have
\begin{theorem} \label{thm:weakCompactness} For $0<T<\infty$, let
$u_n=u_n^r\EE_r+u^z_n\EE_z, d_n=\sin\phi_n \EE_r+\cos\phi_n\EE_z:
\Omega\times (0,T)\to \R^3\times\mathbb S^2$ be a sequence of axisymmetric, without swirl, weak solutions of \eqref{momentum}--\eqref{director}, subject to the initial value $(u_{0n}, d_{0n})=(u_{0n}^r\EE_r
+u_{0n}^z \EE_z, d_{0n}^r\EE_r
+d_{0n}^z \EE_z)\in L^2_{\rm{div}}(\Omega)\times H^1(\Omega,\mathbb S^2)$ and the boundary value  $(u_n, d_n)=(0, d_{0n})$ on $\partial\Omega$. 
Suppose $(u_{0n},d_{0n})\rightarrow (u_0, d_0)$ in $L^2(\Omega)\times H^1(\Omega,\mathbb S^2)$, and
$(u_n, d_n)\rightharpoonup (u,d)$ in $L^2_tL^2_x\times L^2_tH^1_x(\Omega\times [0,T])$. Then $(u,d)$ is an axisymmetric without swirl weak solution 
of \eqref{momentum}--\eqref{director} with initial-boundary value
$(u_0, d_0)$.
\end{theorem}

We would like to sketch the main ideas behind the proof of Theorem \ref{thm:weakExistence}. For the approximation of solutions, we rely on the already mentioned Ginzburg-Landau approximation (see \eqref{ginzburgLandau} below). As $\eps$ tends to zero, the main challenge is that the sequence of Ericksen stress tensors $(\nabla d_\varepsilon\odot \nabla d_\varepsilon-\frac12|\nabla d_\varepsilon|^2\mathbb I_3)$ on the right-hand side of the momentum equation $\eqref{ginzburgLandau}_1$ may not converge weakly in $L^1$. In fact, by the energy law \eqref{energyLawGL}, we are only guaranteed weak convergence as Radon measures for Ericksen stress tensors. To overcome this difficulty in the axisymmetric setting,  we consider two different cases: 
\begin{enumerate}
\item Away from the $z$-axis $\ZZ$, equation \eqref{HMF} already indicates that better regularity is available for the angle function $\varphi$. The situation is more complicated for the approximation. Still, using a blow-up argument, we are able to rule out any concentration of $(e(d_\eps))_\eps$. While some basic properties of the blow-up scheme developed by \cite{lin1999} and \cite{lin2016} play an important role,
we make the crucial new observation that any generic blow-up limit near a possible concentration point off the $z$-axis {\it would be} a nontrivial harmonic map from $\mathbb R^2$ to $\mathbb S^1$ with finite Dirichlet energy, which is impossible. In principle, this is due to  topological restrictions for the blow-up away from the $z$-axis. The details can be found in Subsection \ref{subsection:noConcentrationAwayZ}. 
\item
Right at the $z$-axis $\ZZ$, it is possible that the energy density might concentrate to a limit measure supported on $\ZZ$. To handle this, we employ the idea of concentration-cancellation developed in \cite{kortum2020}. This kind of ideas was first developed by DiPerna-Majda \cite{diperna1988} and Delort \cite{delort1991} in their studies of the Euler equations, see also \cite{majda2002}.  Although the convergence cannot be improved, the ultimate goal is to verify Definition \eqref{defWeakSolution} for the limit quantities. It turns out to be very helpful to use  divergence-free test functions without swirl. Compared to \cite{kortum2020} where the two-dimensional case is considered, the three-dimensional no-swirl functions reduce to  poloidal functions (see e.g.\ \cite{liu2009,schmitt1992}). This fact is discussed in Section \ref{subsection:RadialWeakFormulation} (c.f.\ Remark \ref{remark:Concentration}). The $\EE_r$-component of the right-hand side in \eqref{momentum} might concentrate, but the action of limiting measure on poloidal test functions is zero. For the $\EE_z$-component, we perform a cancellation analysis similar to \cite{evans1990,frehse1982}. Just dealing with poloidal functions and the specific axisymmetry enables us to do so. The conclusion of the argument takes place in Section \ref{subsection:ProofThm}.
\end{enumerate}

The paper is organized as follows. In Section two, we will present the proof of Theorem \ref{thm:weakExistence}. The proof of Theorem \ref{thm:weakCompactness} will be given in Section three.

\section{Convergence of the Ginzburg-Landau approximation}  \label{section:GL}

The following section is devoted to the proof of Theorem \ref{thm:weakExistence}. In order to do so, we consider  the Ginzburg-Landau system for $(u_\varepsilon, P_\varepsilon, d_\varepsilon):\Omega\times [0,\infty)\to 
\mathbb R^3\times\mathbb R\times\mathbb S^2$:
\begin{align}
	\begin{cases}
	  \displaystyle \pa_t u_\eps + (u_\eps \cdot \nabla ) u_\eps - \Delta u_\eps + \nabla P_\eps  = - \div \left( \nabla d_\eps \odot \nabla d_\eps - \frac{|\nabla d_\eps|^2}{2} \mathbb{I}_3 \right), \\
		\displaystyle \div u_{\eps} =0, \\
	 \displaystyle \pa_t d_\eps + (u_\eps \cdot \nabla) d_\eps - \Delta d_\eps = \frac{1 - |d_\eps|^2}{\eps^2} d_\eps
	\end{cases}
	\label{ginzburgLandau}
\end{align}
for $\eps \to 0^+$. The system  is accompanied by the initial and boundary conditions
\begin{align}
	& \begin{cases} (u_\eps(x,0),d_\eps(x,0)) 	 =  (u_{0}(x), d_{0}(x)),
	& x \in \Omega, \\
	  \div  u_{0} = 0, 	& x \in \Omega,   
	\end{cases} \label{initialDataGL} \\
 	& \begin{cases} 	\displaystyle u_{\eps}(x,t) =0, & x \in \pa \Omega,~ t>0, \\
	 \displaystyle d_\eps (x,t) = d_{0}(x), & x \in \pa \Omega,~ t>0,
	\end{cases} \label{boundaryDataGL}
\end{align}
for a given set of initial data $(u_0,d_0) \in L^2_{\div}(\Omega) \times H^{1}(\Omega, \Sp^2)$, 
%in \eqref{iniData}, 
that is assumed to be axisymmetric with no swirl.

Existence of global Leray-Hopf type weak solutions to the above system for a fixed $\eps>0$ has been showed in \cite{lin1995}. It turns out that, thanks to the preservation of the axisymmetricity without swirl by these equations, the construction given by \cite{lin1995} can be modified to obtain a global weak solution
of \eqref{ginzburgLandau}, 
\eqref{initialDataGL} and \eqref{boundaryDataGL},
that is axisymmetric without swirl.
Namely, 
$$u_\eps (r,
\theta,z,t) = u^r_\eps (r,z,t) \EE_r + u^z_\eps (r,z,t) \EE_z, \quad d_\eps (r,\theta, z, t) = d^r_\eps (r,z, t) \EE_r + d^z_\eps (r,z, t) \EE_z$$ for $t>0$. We will provide a sketch of the construction in the appendix below.

 For a given $(u_0,d_0) \in L^2_{\div} \times W^{1,2}(\Omega, \Sp^2)$, the geometry poses some restrictions on the approximability w.r.t.\ to $\eps$ (c.f.\ \cite{bethuel1991}). However, we may take sequences of smooth functions such that
\begin{align}
  (u_{0\eps} , d_{0\eps}) \quad \to  \quad (u_0, d_0)  & \quad \text{ strongly in } L^2_{\div} \times W^{1,2}(\Omega, \R^3)  \label{wellpreparedData} 
 \end{align}
 with $d_{0\eps}|_{r=0} =0$ such that the energy density $(e(d_\eps))_\eps$ strongly converges as well.  
Observe that for a given set of initial data $(u_0,d_0) \in L^2_{\div}(\Omega) \times H^{1}(\Omega, \Sp^2)$ in \eqref{iniData}, that is axisymmetric with no swirl,   
we can take sequences of smooth functions $(u_{0\eps}^r, u_{0\eps}^z)(r,z)\in C^\infty(\mathbb D,\mathbb R^2)$
and $(d_{0\eps}^r, d_{0\eps}^z)(r,z)\in C^\infty(\mathbb D,\mathbb R^2)$, where 
$\mathbb D=\Omega\cap\{\theta=0\}\subset\mathbb R^2$,
such that 
$u_{0\eps}=u_{0\eps}^r\EE_r+u_{0\eps}^z\EE_z$ is divergence free in $\Omega$, and
\begin{equation} \label{wellpreparedData1}
u_{0\eps}\rightarrow u_0 \quad  \text{ strongly in } L^2(\Omega,\mathbb R^3),
\end{equation}
and
\begin{align}
\label{wellpreparedData2}
d_{0\eps}= d_{0\eps}^r\EE_r+d_{0\eps}^z\EE_z
\rightarrow d_0 \quad \text{ strongly in } H^1(\Omega, \mathbb R^3).
\end{align}  
The corresponding axisymmetric solutions $(u_\eps, d_\eps, P_\eps)$ of \eqref{ginzburgLandau}, subject to the initial and boundary condition \eqref{initialDataGL} and \eqref{boundaryDataGL}, satisfy the following energy inequality:
\begin{align}
	\int_\Omega \left[\frac12|u_\eps|^2+e_\eps (d_\eps)\right](\cdot,t)\,\d x &+ 
	\int_0^t \int_\Omega \left[ |\nabla u_\eps |^2 + \left| \pa_t d_\eps + (u_\eps \cdot \nabla ) d_\eps \right|^2 \right] \,\d x \d t\nonumber\\
 &\leq  \int_\Omega \left[\frac12|u_{0}|^2+\frac12 |\nabla d_{0}|^2\right]\d x,   \label{energyLawGL}
\end{align}
for every $t>0$. Here, 
$$e_\eps (d_\eps) := \frac{|\nabla d_\eps|^2}{2}+ \frac{(1-|d_\eps|^2)^2}{4\eps^2}$$ is the Ginzburg-Landau energy density function for $d_\eps$. We refer readers to Chen-Struwe \cite{chen1989b} for the application of this kind of Ginzburg-Landau approximation scheme in the context of heat flow of harmonic maps.

\medskip
In the following, we prove that the limit of solutions (sub-)converges to a weak solution of \eqref{momentum}--\eqref{director} with preserved axisymmetry with no swirl. In order to cope with the most problematic term, the right-hand side of \eqref{ginzburgLandau}, we first  rule out possible energy concentrations \text{away} from $\ZZ$, that ensures $L^1$-convergence away from $\ZZ$ for the Ericksen energy stress tensors,  as $\eps \to 0^+$ at ``good" time slices. In a second step, we realize that a possible concentration of Ginzburg-Landau energy at the $z$-axis $\ZZ$ can be bypassed in the limit process as $\eps \to 0$ due to the axisymmetry and a concentration-cancellation technique. Throughout the proof, we refer several times to \cite{lin2016} for preliminary considerations used in this paper.

\subsection{No concentration away from the $z$-axis}
\label{subsection:noConcentrationAwayZ}
Using the cylindrical symmetry of $d_\eps$, we decompose the approximated director field  $d_\eps = \rho_\eps \cdot \frac{d_\eps}{\rho_\eps}$, where $\rho_\eps=|d_\eps|$, and more precisely 
$$d_\varepsilon(r,\theta, z,t)=\rho_\eps(r,z,t)\left[\sin\phi_\eps(r,z,t) \EE_r+\cos\phi_\eps(r,z,t) \EE_z\right],
\ (r,\theta,z)\in\Omega,\  t>0.$$ 
The density satisfies $0\leq \rho_\eps \leq 1$ by a maximum principle (see e.g.\ \cite[Lemma 2.1]{lin2016} or \cite{bethuel1994}). Hence from \eqref{ginzburgLandau}$_3$ the functions  $(\rho_\varepsilon, \phi_\varepsilon)$ solve
\begin{equation}\label{axisym1}
\begin{cases}\displaystyle
\partial_t\rho_\varepsilon+ (u_\eps \cdot \nabla ) \rho_\eps 
-L\rho_\varepsilon+\rho_\varepsilon \left(|\nabla\phi_\varepsilon|^2+\frac{\sin^2\phi_\varepsilon}{r^2}\right)=\frac{1-\rho_\varepsilon^2}{\varepsilon^2}\rho_\varepsilon,\\
\displaystyle\partial_t\phi_\varepsilon+ (u_\eps \cdot \nabla ) \phi_\eps 
-L\phi_\varepsilon+\frac{\sin\phi_\varepsilon\cos\phi_\varepsilon}{r^2}-2\frac{\nabla\rho_\varepsilon\cdot\nabla\phi_\varepsilon}{\rho_\varepsilon}=0,
\end{cases}
\end{equation}
where $\displaystyle L=\partial^2_r+\frac{1}{r}\partial_r+\partial^2_z$.

The goal is to show that for any  ``good" time slice $t_0>0$, the energy concentration set $\Sigma_{t_0}$ for $e_\eps(d_\eps)$ satisfies $\Sigma_{t_0}\subset\ZZ$. Here we follow some of the techniques developed by \cite{du2022,kortum2020,lin2016} where one considers a time $t_0>0$ such that the dissipation rate in \eqref{energyLawGL} is finite.
Similar to Lin-Wang \cite{lin2016}, according to Fatou's lemma, one has that for any $0<T<\infty$,
$$
\int_0^T\liminf_{\varepsilon\to 0}\int_\Omega |\partial_t d_\varepsilon+u_\varepsilon\cdot\nabla d_\varepsilon|^2(x,t)\d x \d t<+\infty.
$$

Next we consider the subset of good $\Lambda$-time slices. For any positive constant $\Lambda>0$, define $\mathcal{G}_\Lambda^T\subset (0,T)$ by letting
\begin{equation}\label{goodtime}
\mathcal{G}_\Lambda^T=\Big\{t_0\in (0, T) \ | \ 
\Lambda(t_0):=\liminf_{\varepsilon\to 0} \int_\Omega |\tau_\varepsilon(t_0)|^2\,dx\le \Lambda\Big\},
\end{equation}
where
$$
\tau_\varepsilon(t_0)(\cdot)=(\partial_t d_\varepsilon+u_\varepsilon\cdot\nabla d_\varepsilon)(\cdot, t_0).
$$
Set $\mathcal{B}_\Lambda^T=(0,T)\setminus \mathcal{G}_\Lambda^T$.
Then it is easy to see that
\begin{equation}\label{sizeofbadtime}
    \big|\mathcal{B}_\Lambda^T\big|\le \frac{E_0}{\Lambda},
\end{equation}
where 
$$E_0=\sup_{0<\varepsilon\le 1} \int_\Omega \left(\frac12 |u_{0\eps}|^2+e_\eps(d_{0\eps})\right)\,\d x.$$
From Lin-Wang \cite{lin2016}, we additionally have
\begin{theorem}    \label{thm:epsilonCompactness}
Let $(u_\eps,d_\eps)$ be a solution (not necessarily axisymmetric) to \eqref{ginzburgLandau} satisfying the energy inequality \eqref{energyLawGL}. For any $\Lambda>0$, let $t_0\in\mathcal{G}_\Lambda^T$. Then there exist $\varepsilon_0>0$ and $r_0>0$ such that if we define the concentration set for $d_\eps$
$$
\Sigma_{t_0}: =\bigcap_{0<r<r_0}\Big\{x_0\in\Omega: \liminf_{\varepsilon\to 0} r^{-1}\int_{B_r(x_0)} e_\varepsilon(d_\varepsilon)(x,t_0)>\varepsilon_0^2\Big\},
$$
then it holds 
\begin{itemize}
\item [i)]
$\Sigma_{t_0}$ is a closed subset with $H^1(\Sigma_{t_0}\cap K)<+\infty$ for any compact subset $K\subset\Omega$, and
\item [ii)] there exists an approximated harmonic map $d:\Omega\to\mathbb S^2$ with tension field
$\displaystyle\tau=\lim_{\varepsilon\to 0}\tau_{\varepsilon}(\cdot, t_0)$ weakly in $L^2(\Omega)$ such that, after passing to a subsequence, 
$$
d_\varepsilon(\cdot,t_0)\rightarrow d \ \ {\rm{in}}\ \ H^1_{\rm{loc}}(\Omega\setminus\Sigma_{t_0}),
$$
and
$$
e_\varepsilon(d_\varepsilon)(\cdot,t_0)\,dx\rightharpoonup \frac12|\nabla d|^2\,dx
$$ 
as convergence of Radon measures on $\Omega\setminus\Sigma_{t_0}$. 
\end{itemize}
\end{theorem}

In the following we will show 
\begin{theorem} \label{offZaxis}Under the same notations as above. For any $\Lambda>0$ and $t_0\in \mathcal{G}_\Lambda^T$, it holds that
\begin{equation}\label{offZaxis1}
  \Sigma_{t_0}\setminus\ZZ=\emptyset.  
\end{equation}  
\end{theorem}  

\pf  Since $d_\eps$
is axisymmetric, it follows that $\Sigma_{t_0}=\mathcal{S}_{t_0}\times \mathbb S^1$, where $\mathcal{S}_{t_0}: =\Sigma_{t_0}\cap\mathbb D$
and $\mathbb D=\big\{(r,0,z)\in\Omega\big\}$.
Since $H^1(\Sigma_{t_0}\cap K)<+\infty$ for any compact set $K\subset\mathbb R^3$, it follows that if we picked up any $(r_0, z_0)\in\mathcal{S}_{t_0}$, with $r_0>0$, there would exist a positive integer
$N_0$ such that 
$$\# (\mathcal{C}_{\frac{r_0}2}^2(r_0,z_0)\cap \mathcal{S}_{t_0})=N_0,$$ 
where 
$$\mathcal{C}_{\delta}^2(r_0,z_0):=\big\{(r,z): \ |r-r_0|<\delta, \ |z-z_0|<\delta\big\}, \ 0<\delta\le r_0,$$
is the 2-dimensional poly-disc with center $(r_0,z_0)$ and  radius $\delta$.  
Thus there would exist $0<r_1<\frac{r_0}2$ such that 
\begin{equation}\label{isolated}
\mathcal{C}_{r_1}^2(r_0,z_0)\cap \mathcal{S}_{t_0}=\big\{(r_0, z_0)\big\}.
\end{equation}

Define the 3-dimensional poly-disc $\mathcal{C}^3_\delta(r_0,0,z_0)=\big\{(r,\theta, z):\ (r,z)\in \mathcal{C}^2_\delta(r_0,z_0), \ -\delta<\theta<\delta\big\}$. 
It is easy to see that there exist absolute constants $0<c_1<c_2$ such that 
\begin{equation}\label{equiv}
B_{c_1\delta}(r_0,0,z_0)\subset \mathcal{C}^3_\delta(r_0,0,z_0)\subset B_{c_2\delta}(r_0,0,z_0),
\ \forall 0<\delta\le \frac{r_0}4.
\end{equation}
Direct calculations imply
$$
e_\varepsilon(d_\varepsilon)=\frac12\left(|\nabla_{r,z}\rho_\varepsilon|^2+\rho_\varepsilon^2|\nabla_{r,z}\phi_\varepsilon|^2
+\frac{\rho_\varepsilon^2\sin^2\phi_\varepsilon}{r^2}\right)+\frac{(1-\rho_\varepsilon^2)^2}{4\varepsilon^2} =: e_\varepsilon(\rho_\varepsilon, \phi_\varepsilon),
$$
where $\nabla_{r,z} f(r,z)=(\partial_r f, \partial_z f)(r,z)$. From the definition of $\Sigma_{t_0}$ and \eqref{equiv}, we have that there exist $0<\alpha_2<\frac12$, $\alpha_1>0$ such that
$(r_0, 0, z_0)\in \Sigma_{t_0}$ if and only if 
\begin{equation}\label{alt-def}
\liminf_{\varepsilon\to 0}\frac{1}{2\delta} \int_{\mathcal{C}^3_\delta(r_0,0,z_0)} e_\varepsilon(\rho_\varepsilon, \phi_\varepsilon) r\d r \d \theta \d z>\alpha_1\varepsilon_0^2,
\quad  \forall 0<\delta<\alpha_2 r_0.
\end{equation} 
On the other hand, observe that 
$$
 \int_{\mathcal{C}^3_\delta(r_0,0,z_0)} e_\varepsilon(\rho_\varepsilon, \phi_\varepsilon) r \d r \d \theta \d z=2\delta 
  \int_{\mathcal{C}^2_\delta(r_0,z_0)} e_\varepsilon(\rho_\varepsilon, \phi_\varepsilon) r \d r \d z
$$
so that \eqref{alt-def} is equivalent to
\begin{equation}\label{alt-def1}
\liminf_{\varepsilon\to 0} \int_{\mathcal{C}^2_\delta(r_0,z_0)} e_\varepsilon(\rho_\varepsilon, \phi_\varepsilon) r \d r \d z>\alpha_1\varepsilon_0^2,
\quad  \forall 0<\delta<\alpha_2 r_0.
\end{equation} 
Since $\frac{r_0}{2}\le r\le \frac{3r_0}2$ for any $(r,z)\in \mathcal{C}^2_{\delta}(r_0,z_0)$ with $0<\delta<\alpha_2 r_0\le \frac{r_0}2$, inequality \eqref{alt-def1} is equivalent to
\begin{equation}\label{alt-def2}
\liminf_{\varepsilon\to 0} \int_{\mathcal{C}^2_\delta(r_0,z_0)} e_\varepsilon(\rho_\varepsilon, \phi_\varepsilon) \d r \d z>\frac{\alpha_3\varepsilon_0^2}{r_0},
\quad  \forall 0<\delta<\alpha_2 r_0,
\end{equation} 
for some $\alpha_3>0$ depending only on $\alpha_1$. 
Set $r_2=\min\{r_1, \alpha_2 r_0\}>0$. In order to track the position of the energy concentration, for $0<\lambda<\frac{r_2}2$ we define 
\begin{equation}\label{max_exhau}
\Theta_\varepsilon((r_0,z_0); \lambda)
:=\max\left\{ \int_{\mathcal{C}^2_\lambda(r_*,z_*)} e_\varepsilon(\rho_\varepsilon, \phi_\varepsilon) r \d r \d z:
\, (r_*,z_*)\in \mathcal{C}^2_{\frac{r_2}2}(r_0,z_0)\right\}.
\end{equation}
According to the definition \eqref{max_exhau} for $\Theta_\eps((r_0,z_0),\lambda)$, 
we know that for a sufficiently large constant $C_*>0$ (to be determined later),
there exist $0<\lambda_\varepsilon<\frac{r_2}2$ and $(r_\varepsilon, z_\varepsilon)
\in\mathcal{C}^2_{\frac{r_2}2}(r_0,z_0)$ such that
\begin{align}\label{max_exhau1}
\frac{\varepsilon_0^2}{C_*}=\Theta_\varepsilon((r_0,z_0); \lambda_\varepsilon)
=\int_{\mathcal{C}^2_{\lambda_\varepsilon}(r_\varepsilon,z_\varepsilon)} e_\varepsilon(\rho_\varepsilon, \phi_\varepsilon) r \d r \d z.
\end{align}
Next we claim that
\begin{lemma}
It holds
\begin{itemize}
\item[i)]  $\displaystyle\lim_{\varepsilon\to 0}\lambda_\varepsilon=0,$
\item[ii)]  $\displaystyle\lim_{\varepsilon\to 0}(r_\varepsilon, z_\varepsilon)=(r_0,z_0).$
\end{itemize}
\end{lemma} 
\begin{proof}
Suppose i) was false. Then there would exist $\lambda_0>0$ such that $ \displaystyle\lim_{\varepsilon\to 0}\lambda_\varepsilon=\lambda_0$ 
and hence 
$$
\frac{1}{2\lambda_0}\int_{\mathcal{C}^3_{\lambda_0}(r_0,0,z_0)}e_\varepsilon(\rho_\varepsilon, \phi_\varepsilon)r\d r\d \theta \d z
=\int_{\mathcal{C}^2_{\lambda_0}(r_0,z_0)}e_\varepsilon(\rho_\varepsilon, \phi_\varepsilon)r\d r \d z\le \Theta_\varepsilon((r_0,z_0); \lambda_0)=\frac{\varepsilon_0^2}{C_*}.
$$
This, combined with the $\varepsilon_0$-compactness Theorem \ref{thm:epsilonCompactness}, would imply that $(r_0,0,z_0)\notin \Sigma_{t_0}$, which contradicts to the choice of $(r_0, 0,z_0)$. \\
Suppose ii) was false. Then we would assume that $(r_\varepsilon, z_\varepsilon)\rightarrow (r_*,z_*)\in \mathcal{C}_{\frac{r_2}2}^2(r_0,z_0)$ for some
$(r_*,z_*)\not=(r_0,z_0)$. Hence for any fixed $0<r<\frac{r_2}2$, it holds that 
\begin{align*}
\frac{1}{2r}\liminf_{\eps \to 0}\int_{\mathcal{C}^3_{r}(r_*,0, z_*)} e_\varepsilon(\rho_\varepsilon,\phi_\varepsilon) r\d r \d \theta \d z
&=\liminf_{\eps\to 0}\int_{\mathcal{C}^2_{r}(r_*, z_*)} e_\eps(\rho_\eps,\phi_\eps) r \d r \d z\\
&\ge 
\liminf_{\eps \to 0}\int_{\mathcal{C}^2_{r}(r_\eps, z_\eps)} e_\eps(\rho_\eps,\phi_\eps) r \d r \d z \ge \frac{\varepsilon_0^2}{C_*}.
\end{align*}
This, combined with the definition of $\Sigma_{t_0}$, implies that $(r_*,0,z_*)\in\Sigma_{t_0}\cap \mathcal{C}^3_{r_1}(r_0,0,z_0)$,
which contradicts to
\eqref{isolated}.
\end{proof}

As next step,  we perform a blow-up argument of $(\rho_\varepsilon, \phi_\varepsilon)$ centered at $(r_\varepsilon, z_\varepsilon)$. This procedure is similar to \cite{lin1999}.  
To begin with, we set 
\begin{align*}	
	(q_\eps, \psi_\eps)(r,z):=(\rho_\eps, \phi_\eps)(r_\eps+\lambda_\eps r, z_\eps+\lambda_\eps z): D_\eps \to [-1,1] \times \R,
\end{align*}
where $D_\eps:=\lambda_\eps^{-1}\left(\mathcal{C}^2_{\frac{r_1}2}(r_0,z_0)\setminus\big\{(r_\eps, z_\eps)\big\}\right).$

Direct calculations, combined with \eqref{max_exhau1}, imply 
\begin{align}\label{max_exhau2}
\frac{\varepsilon_0^2}{C_*}&=\int_{\mathcal{C}^2_1(0,0)} \widehat{e}_\varepsilon(q_\varepsilon, \psi_\varepsilon)(r,z) (r_\varepsilon+\lambda_\varepsilon r) \d r \d z \nonumber\\
&\ge \int_{\mathcal{C}^2_1(r,z)} \widehat{e}_\varepsilon(q_\varepsilon, \psi_\varepsilon)(r,z) (r_\varepsilon+\lambda_\varepsilon r) \d r \d z,\quad  \forall (r,z)\in D_\varepsilon,
\end{align}
where 
\begin{align*}
\widehat{e}_\eps & (q_\varepsilon, \psi_\varepsilon)(r,z) \\
& =\frac12\left(|\nabla_{r,z}q_\varepsilon|^2+q_\varepsilon^2|\nabla_{r,z}\psi_\varepsilon|^2
+\frac{\lambda_\varepsilon^2}{(r_\varepsilon+\lambda_\varepsilon r)^2} q_\varepsilon^2\sin^2\psi_\varepsilon\right)(r,z)
+\frac{\lambda_\varepsilon^2(1-q^2_\varepsilon(r,z))^2}{4\varepsilon^2}.
\end{align*}
Moreover, for any $R>0$, by the almost monotonicity inequality from Lin-Wang \cite[Lemma 3.1]{lin2016} and the $\theta$-independence of $\widehat{e}_\eps$, it holds that
\begin{align}\label{total_energy}
\int_{\mathcal{C}^2_R(0,0)} & \widehat{e}_\eps (q_\eps, \psi_\eps)(r,z)(r_\eps+\lambda_\eps r) \d r \d z\nonumber\\
&=(2R)^{-1} \int_{\mathcal{C}^3_R(0,0,0)} \widehat{e}_\eps(q_\eps, \psi_\eps)(r,z)(r_\eps+\lambda_\eps r) \d r \d \theta \d z\nonumber\\
&=(2R)^{-1}\int_{\mathcal{C}^2_{R\lambda_\eps}(r_\eps, z_\eps)\times [-R, R]} e_\eps(d_\eps)\,r\d r \d \theta \d z\nonumber\\
&\le (2R)^{-1}\sum_{j=-[\frac{1}{\lambda_\eps}]-1}^{[\frac{1}{\lambda_\eps}]+1}
\int_{\mathcal{C}^2_{R\lambda_\eps}(r_\eps, z_\eps)\times [j R\lambda_\eps, (j+1)R\lambda_\eps ]} e_\eps(d_\eps)\,r\d r \d \theta \d z \nonumber\\
&\le\lambda_\eps \sum_{j=-[\frac{1}{\lambda_\eps}]-1}^{[\frac{1}{\lambda_\eps}]+1}
(R\lambda_\eps)^{-1}\int_{\mathcal{C}^3_{R\lambda_\eps}(r_\eps, jR\lambda_\eps, z_\eps)} e_\eps(d_\eps)\,r\d r \d \theta \d z\nonumber\\
& \le C(E_0+\Lambda).
\end{align}
Scaling of \eqref{axisym1} implies 
\begin{equation}\label{blowup_eqn1}
\widehat{L}_\eps\psi_\eps-\frac{\lambda_\eps^2\sin\psi_\eps\cos\psi_\eps}
{(r_\eps+\lambda_\eps r)^2}+2\frac{\nabla q_\eps\cdot\nabla\psi_\eps}{q_\eps}=
\tau_\eps^1\ \ \ \ \ \ \ \ \ \ \ \ \ \ \ \ \ \ \ \ {\rm{in}}\ \ D_\eps,
\end{equation}
and 
\begin{equation}\label{blowup_eqn2}
\widehat{L}_\eps q_\eps-q_\eps\left(|\nabla\psi_\eps|^2
+\frac{\lambda_\eps^2\sin^2\psi_\eps}{(r_\eps+\lambda_\eps r)^2}\right)+
\frac{\lambda_\eps^2(1-q_\eps^2)}{\eps^2}q_\eps=\tau_\eps^2 \ \ \ \ {\rm{in}}\ \ D_\eps,
\end{equation}
where
$$
\begin{cases}\displaystyle\widehat{L}_\eps={\partial^2_r}+\frac{\lambda_\eps}{r_\eps+\lambda_\eps r}\partial_ r+
{\partial_ z^2},\\
\displaystyle\tau_\eps^1(r,z)=\lambda_\eps^2(\displaystyle\partial_t\phi_\eps+u_\epsilon^r \partial_r\phi_\eps+u_\eps^z\partial_z\phi_\eps)
(r_\eps+\lambda_\eps r, z_\eps+\lambda_\eps z),\\
\displaystyle\tau_\eps^2(r,z)=\lambda_\eps^2(\partial_t\rho_\eps+u_\epsilon^r \partial_r\rho_\eps+u_\eps^z\partial_z\rho_\eps)(r_\eps+\lambda_\eps r, z_\eps+\lambda_\eps z).
\end{cases} 
$$
From \eqref{goodtime}, we have that
$$
\max\big\{\big\|\tau_\eps^1\big\|_{L^2(D_\eps)}, 
\ \big\|\tau_\eps^2\big\|_{L^2(D_\eps)}\big\}
\le \lambda_\eps \big\|\tau_\eps\big\|_{L^2(\Omega)}\le \lambda_\eps \Lambda ~ \to ~ 0.
$$
At this point, we observe that there are three cases of possible limit scaling:
\begin{align}   \label{limitCases}
	\lim_{ \eps \to 0^+} \frac{\lambda_\eps}{\eps} = \begin{cases} + \infty, \\  \gamma \in \R_+, \\ 0, \end{cases} 
\end{align}
up to a subsequence. In comparison to \cite{lin1999}, the most difficult case is the first one. On the one hand, it would lead to a limiting non-trivial harmonic map to $\mathbb S^1$ if the blow-up actually existed. On the other hand, we show that no non-trivial harmonic map to $\mathbb S^1$ can be the scaling limit.  More precisely, we  have
\begin{lemma}  \label{thm:noBlowUp}
Let $\displaystyle\lim_{ \eps \to 0^+} \frac{\lambda_\eps}{\eps} = +\infty$. Then, up to a subsequence, it holds
\begin{equation}\label{h1conv}
(q_\eps, \psi_\eps)\rightarrow (1,c_0) \ \ {\rm{in}}\ H^1_{\rm{loc}}(\mathbb R^2). 
\end{equation}
for a constant  $c_0\in\mathbb R$.
\end{lemma}
The remainder of this section is devoted to verification of Lemma \ref{thm:noBlowUp} and afterwards the consideration of the remaining two cases in \eqref{limitCases}.
In order to show \eqref{h1conv}, first observe \eqref{max_exhau2} implies that
\begin{align} \label{max_exhau3}
\begin{split}
 \int_{\mathcal{C}^2_1(0,0)} \widehat{e}_\eps(q_\eps, \psi_\eps)(r,z)\, \d r \d z&\ge \frac{\eps_0^2}{2r_0C_*},  \\
\int_{\mathcal{C}^2_1(r,z)} \widehat{e}_\eps(q_\eps, \psi_\eps)(r,z) \,\d r \d z&\le \frac{2\eps_0^2}{r_0C_*},
\quad  \forall (r,z)\in D_\eps. 
\end{split}
\end{align}
From \eqref{max_exhau3} and the fact $\frac{\lambda_\eps}{\eps} \to +\infty$, we may assume that  (c.f.\ \cite[Proof of Lemma 4.1]{lin2016})
\begin{align}
    |q_\eps|\ge \frac12 \ \ {\rm{in}}\ \ D_\eps.  \label{boundFromBelow}
\end{align}
For any $(r_*, z_*)\in D_\eps$, with ${\rm{dist}}( (r_*, z_*), \partial D_\eps)\ge 3$, we take a localization, i.e.\ let $\eta\in C_0^\infty(\mathcal{C}^2_2(r_*,z_*))$ be such that $0\le \eta\le1$, and $\eta\equiv 1$ in
$\mathcal{C}^2_1(r_*,z_*)$. Then one can verify that $v_\eps=\eta^2\big(\psi_\eps-\overline{\psi_\eps}\big)$, where 
$\overline{\psi_\eps}=\displaystyle\fint_{\mathcal{C}^2_2(r_*,z_*)}\psi_\eps$, satisfies
\begin{eqnarray}\label{blowup3}
(\partial^2_r+\partial^2_z)v_\eps &=&-2\frac{\nabla q_\eps}{q_\eps} \cdot\nabla v_\eps
+\eta^2\tau_\eps^1+2\frac{\nabla q_\eps\cdot\nabla \eta^2}{q_\eps} \cdot (\psi_\eps-\overline{\psi_\eps})\nonumber\\
&&+\frac{\lambda_\eps^2}{(r_\eps+\lambda_\eps r)^2} \eta^2 \sin\psi_\eps \cos\psi_\eps
-\frac{\lambda_\eps}{r_\eps+\lambda_\eps r}\eta^2\partial_r \psi_\eps\nonumber\\
&&+(\psi_\eps-\overline{\psi_\eps})(\partial^2_r+\partial^2_z)\eta^2+2\nabla_{r,z}\eta^2\cdot\nabla_{r,z}\psi_\eps.
\end{eqnarray}
Applying the $W^{2,\frac43}$-estimate for the Poisson equation (cf. \cite{gilbarg2001}), we obtain that 
\begin{align*}
&\big\|\nabla_{r,z} v_\eps\big\|_{L^4(\mathbb R^2)}\le C\big\|\nabla^2_{r,z} v_\eps\big\|_{L^\frac43(\mathbb R^2)}\\
&\le C\Big(\big\|\nabla_{r,z} q_\eps\big\|_{L^2(\rm{supp}(\eta))}\big(\big\|\nabla_{r,z}v_\eps\big\|_{L^4(\mathbb R^2)}
+\big\|\psi_\eps-\overline{\psi_\eps}\big\|_{L^4(\mathcal{C}^2_2(r_*,z_*))}\big)
+\big\|\eta^2 \tau_\eps^1\big\|_{L^2(\mathbb R^2)}\\
&\ \ +\frac{\lambda_\eps^2}{r_0^2} +\big\|\nabla_{r,z}\psi_\eps\big\|_{L^2(\mathcal{C}^2_2(r_*,z_*))}
\Big).
\end{align*}
This implies
\begin{align*}
&\big(1-C\big\|\nabla_{r,z}q_\eps\big\|_{L^2(\mathcal{C}^2_2(r_*,z_*))}\big) \big\|\nabla_{r,z} v_\eps\big\|_{L^4(\R^2)}\\
&\le C\Big(\big\|\nabla_{r,z} q_\eps\big\|_{L^2(\mathcal{C}^2_2(r_*,z_*)}\big\|\nabla_{r,z}\psi_\eps\big\|_{L^2(\mathcal{C}^2_2(r_*,z_*))}
+(\Lambda +r_0^{-2}) \lambda_\eps^2+\big\|\nabla_{r,z}\psi_\eps\big\|_{L^2(\mathcal{C}^2_2(r_*,z_*)}\Big)\\
&\le C\Big(r_0^{-1}\eps_0^2+(\Lambda+r_0^{-2}) r_0^2\Big).
\end{align*} 
Now we choose $C_*\ge \frac{4C\eps_0^ 2}{r_0}$  in \eqref{max_exhau1} so that
$C\big\|\nabla_{r,z}q_\eps\big\|_{L^2(\mathcal{C}^2_2(r_*,z_*))}\le \frac{C\eps_0^2}{r_0C_*}\le \frac12.
$
Hence
\begin{equation}\label{l4-bound}
\big\|\nabla_{r,z} \psi_\eps\big\|_{L^4(\mathcal{C}^2_1(r_*,z_*))} \le\big\|\nabla_{r,z} v_\eps\big\|_{L^4}
\le 2C\Big(r_0^{-1}\eps_0^2+(\Lambda+r_0^{-2}) r_0^2\Big)
\end{equation} 
holds for any $(r_*,z_*)\in D_\eps$ with ${\rm{dist}}((r_*, z_*), \partial D_\eps)\ge 3$. 

Observe that $D_\eps\to\mathbb R^2$ as $\eps\to 0$. Combining this with \eqref{l4-bound}, we may assume that there exists
a function $\psi\in W^{1,4}_{\rm{loc}}(\mathbb R^2)$ such that $\psi_\eps\to\psi$ in $W^{1,p}_{\rm{loc}}(\mathbb R^2)$ for $2\le p<4$. Also observe that
\eqref{max_exhau2} in combination with $\frac{\lambda_\eps} {\eps} \to +\infty$ implies that $q_\varepsilon\rightharpoonup 1$ in $H^1(\mathbb R^2)$ and $q_\eps\rightarrow 1$ a.e.\ and in $L^q(\R^2)$ locally in
$\mathbb R^2$ for any $1\le q<\infty$. Hence, using \eqref{boundFromBelow} and the Dominated Convergence Theorem, 
$$\frac{\nabla\rho_\eps\cdot\nabla\psi_\eps}{\rho_\eps}\rightharpoonup 0 \ \ {\rm{in}}\ \ L^1_{\rm{loc}}(\mathbb R^2). %\todo{(locally?)}
$$
Therefore, after passing $\eps\to 0$ in  equation \eqref{blowup_eqn1}, we obtain that
$$
(\partial_r^2+\partial_z^2)\psi=0\  \ \ {\rm{in}}\ \  \mathbb R^2,
$$
for the weak limit $ \psi$ of $(\psi_\eps)_\eps$ in $H^1(\R^2)$. Namely, $\psi$ is a harmonic function in $\mathbb R^2$. On the other hand, it follows from \eqref{total_energy} that
$$
\int_{\mathbb R^2}|\nabla_{r,z}\psi|^2\,\d r \d z\le C(E_0+\Lambda).
$$
Therefore, $\psi=c_0$ is a constant function. In particular, one obtains that
\begin{equation}\label{psi-h1conv}
\psi_\eps\rightarrow c_0 \quad  {\rm{and}} \quad  \nabla\psi_\eps\to 0 \ \ {\rm{in}}\ \ L^2_{\rm{loc}}(\mathbb R^2). 
\end{equation}
Next we want to show that 
\begin{equation}\label{q-h1conv}
q_\eps \to 1 \ \ {\rm{in}}\ \ H^1_{\rm{loc}}(\mathbb R^2).
\end{equation}
Denote $f_\eps=1-q_\eps$. It follows
from equation \eqref{blowup_eqn2} that $f_\eps$ solves
\begin{equation}\label{blowup_eqn4}
-(\partial_r^2+\partial_z^2)f_\epsilon+
\frac{\lambda_\eps^2(1+q_\eps)q_\eps}{\eps^2}f_\eps
=q_\eps\left(|\nabla\psi_\eps|^2+\frac{\lambda_\eps^2\sin^2\psi_\eps}{(r_\eps+\lambda_\eps r)^2}\right)+
\tau_\eps^2 \qquad  {\rm{in}}\ \ D_\eps.
\end{equation}
Multiplying both sides of \eqref{blowup_eqn4} by $\eta^2 f_\eps$, where $\eta$ is the cut-off function of $\mathcal{C}^2_1(r_*,z_*)$ as above,
and integrating over $\mathcal{C}^2_2(r_*,z_*)$, we obtain that
\begin{align*}
&\int_{\mathcal{C}^2_2(r_*,z_*)} \big(\eta^2|\nabla_{r,z} f_\eps|^2+\frac{\lambda_\eps^2(1+q_\eps)q_\eps}{\eps^2}\eta^2f_\eps^2\big)\,\d r \d z\\
&=\int_{\mathcal{C}^2_2(r_*,z_*)}\left[\eta^2\left( q_\eps\left(|\nabla\psi_\eps|^2+\frac{\lambda_\eps^2\sin^2\psi_\eps}{(r_\eps+\lambda_\eps r)^2}\right)+
\tau_\eps^2\right)f_\eps-f_\eps\nabla_{r,z}f_\eps\cdot\nabla_{r,z}\eta^2\right]\, \d r \d z\\
&\le C\big\|f_\eps\big\|_{L^2(\mathcal{C}^2_2(r_*,z_*))}\left(1+\big\|\eta \nabla \psi_\eps\big\|_{L^4(\mathcal{C}^2_1(r_*,z_*))}^2
+\lambda_\eps\big\|\tau_\eps\big\|_{L^2(\mathcal{C}^2_2(r_*,z_*))}\right) \\
& \spc +\frac12\int_{\mathcal{C}^2_2(r_*,z_*)} \eta^2|\nabla_{r,z} f_\eps|^2\,\d r \d z.
\end{align*}
This yields that for any $(r_*,z_*)\in D_\eps$, we have that
\begin{align*}
&\int_{\mathcal{C}^2_2(r_*,z_*)} \eta^2\left(|\nabla_{r,z} q_\eps|^2+\frac{\lambda_\eps^2}{\eps^2}(1-q_\eps)^2\right)\,\d r \d z\\
&\le\int_{\mathcal{C}^2_2(r_*,z_*)} \eta^2\big(|\nabla_{r,z} f_\eps|^2+\frac{\lambda_\eps^2}{\eps^2}f_\eps^2\big)\,drdz\\
&\le C\eps \no{\frac{1-q^2_\eps}{\eps}}{L^2(\mathcal{C}^2_2(r_*,z_*))}\left(1+\big\|\eta \nabla \psi_\eps\big\|_{L^4(\mathcal{C}^2_2(r_*,z_*))}^2
+\lambda_\eps\big\|\tau_\eps\big\|_{L^2(\mathcal{C}^2_2(r_*,z_*))}\right) \\
&\to 0.
\end{align*} 
This implies \eqref{q-h1conv} as well
\begin{equation}\label{potential-conv}
\frac{\lambda_\eps^2}{\eps^2}(1-q_\eps)^2\to 0 \ \ \ {\rm{in}}\ \ \ L^1_{\rm{loc}}(\mathbb R^2). 
\end{equation}
Combining \eqref{potential-conv} and \eqref{q-h1conv} with \eqref{psi-h1conv}, we obtain that
\begin{align*}
&\frac{\eps_0^2}{C_*}=\int_{\mathcal{C}^2_1(0,0)} \widehat{e}(q_\eps, \psi_\eps)(r,z) (r_\eps+\lambda_\eps r)\,\d r \d z\\
&\le 2r_0\int_{\mathcal{C}^2_1(0,0)} \left[|\nabla_{r,z}q_\eps|^2+q_\eps^2|\nabla_{r,z}\psi_\eps|^2
+\frac{\lambda_\eps^2q_\eps^2}{(r_\varepsilon+\lambda_\eps r)^2}\sin^2\psi_\eps+\frac{\lambda_\eps^2}{\eps^2}(1-q_\eps^2)^2\right](r,z)\, \d r \d z \\
&\rightarrow 0.
\end{align*}
This is clearly impossible.   

\medskip
Concerning the remaining cases of \eqref{limitCases}, we refer to \cite{lin1999} for details and just briefly sketch them here. 
It suffices to consider $\eqref{ginzburgLandau}_3$. 
wW define as above 
\begin{align*}
w_\eps(r,z)=w^r_\eps(r,z)\EE_r+w^z_\eps(r,z)\EE_z:=d_\eps (r_\eps+\lambda_\eps r, z_\eps+\lambda_\eps z, t_0): D_\eps \to \R^3.
\end{align*}
Then
\begin{align}\label{blowupEqn2.0}
		&-\big(\frac{\partial^2}{\partial r^2}
  +\frac{\lambda_\eps}{r}\frac{\partial}{\partial r}+\frac{\partial^2}{\partial z^2}\big)w_\eps^r -\frac{\lambda^2_\eps}{r^2}w_\eps^r+\frac{\lambda_\eps^2}{\eps^2 } (1-|w_\eps|^2) w_\eps^r = \tilde{\tau}_\eps^r,\\
&\quad-\big(\frac{\partial^2}{\partial r^2}
  +\frac{\lambda_\eps}{r}\frac{\partial}{\partial r}+\frac{\partial^2}{\partial z^2}\big)w_\eps^z 
  +\frac{\lambda_\eps^2}{\eps^2 } (1-|w_\eps|^2) w_\eps^z = \tilde{\tau}_\eps^z,
\end{align}
in $D_\eps$, with 
$$\tilde{\tau}_\eps(r,z)=\tilde{\tau}_\eps^r(r,z)\EE_r
+\tilde{\tau}_\eps^z(r,z)\EE_z:= \lambda_\eps^2 \left( \pa_t d_\eps + (u_\eps \cdot \nabla) d_\eps \right) (r_\eps + \lambda_\eps r, z_\eps + \lambda_\eps z, t_0).$$ 
In the case that $\frac{\lambda_\eps}{\eps} \to \gamma>0$, recalling \eqref{max_exhau1} and \eqref{goodtime}, 
we have that $(w_\eps)_\eps$ converges locally in $H^1(\R^2)$ to a limiting map $ w$ satisfying 
\begin{align*}
	- \Delta w +  \gamma^2 (1-|w|^2) w =0,		\quad \text{on }\ \  \R^2.
\end{align*}
This, in combination with 
$$\frac{\eps_0^2}{C^*} \le \int_{\R^2} \left[ \frac12{|\nabla w|^2} + \gamma^2(1-|w|^2)^2\right] \d r \d z<\infty,$$
implies $w$ must be nontrivial.
On the other hand, by a Pohozaev argument as in \cite{lin1999} such a map must be constant. We get a desired contradiction. 

The final case that $\frac{\lambda_\eps}{\eps} \to 0$ leads to a limiting map $w$ satisfying 
\begin{align*}
		\Delta w = 0 \qquad \text{on }\ \ \R^2,
\end{align*}
and
$$
0<\int_{\R^2}|\nabla w|^2\,drdz<+\infty.
$$
Such a harmonic map must be constant as well, which is again impossible.

\smallskip
Putting all these three cases together, we proved
that $\Sigma_{t_0}\setminus\ZZ=\emptyset$. \qed

\bigskip
With Theorem \ref{offZaxis} at hand, we can follow exactly the proof of \cite[Theorem 1.1, pages 1563-1567]{lin2016} to verify that the 
 weak limit $(u, d)$ of $(u_\eps, d_\eps)$ in $L^2_tL^2_x \times L^2_tH^1_x (\Omega\times [0,T])$, for all $0<T<\infty$, is a weak solution of
the Ericksen-Leslie system \eqref{momentum}--\eqref{director} over $(\Omega\setminus\ZZ)\times (0, \infty)$. 
Moreover, since the property of axisymmetric without swirl is preserved under this convergence, we have that $(u, d, P)$ is axisymmetric without swirl in $\Omega\times (0,\infty)$.

\subsection{Possible concentration at the $z$-axis}

\label{subsection:RadialWeakFormulation}
We investigate the weak formulation in Definition \ref{defWeakSolution} in the case of axisymmetry without swirl. The symmetry does not only improve the regularity properties of solutions $(u,d)$ to \eqref{momentum}--\eqref{director} (and $(u_\eps,d_\eps)$ to \eqref{ginzburgLandau} respectively) away from the origin but also restricts the class of \textit{effective test functions}. 
In conclusion, we are left with test functions for Definition \ref{defWeakSolution} given by
$$ \phi = \phi^r(r,z,t) \EE_r + \phi^z(r,z,t) \EE_z,$$
with smooth coefficients $\phi^r,\phi^z$, and $\phi^r(0,z,t)=0$ for all $(0, z) \in \mathbb D$.  This elementary observation is very crucial in the following as possible concentrations appearing in the limit $\eps \to 0$  do not have an impact for the limiting weak solution extended across the $z$-axis.
The divergence-free condition further restricts the structure of test functions. It implies the existence of a smooth stream function $\psi: \Omega \to \R$ (c.f.\ \cite{liu2009}) such that 
\begin{align} \label{streamFunction}
    \phi^r = - \frac{\pa_z \psi}{r}, \qquad \phi^z = \frac{\pa_r  \psi}{r}.
\end{align} 
Hence we are left the verification of \eqref{weakMomentum} just for such functions simplifying the upcoming argument.

\begin{remark}  \label{remark:Concentration}
The key observation that $\phi^r(r|_{r=0},z,t)=0$ implies that the concentration is not seen by test functions applies from another point of view directly to $(d_\eps)_\eps$. As equations  \eqref{momentum}--\eqref{director} and \eqref{ginzburgLandau} only have $\EE_r$ and $\EE_z$ components, we have
\begin{align}
    &\div  \left(\nabla d_\eps \odot \nabla d_\eps - e_\eps (d_\eps) \mathbb{I}_3 \right)    =     (\nabla d_\eps )^\top \left( \Delta d_\eps  + \frac{1-|d_\eps|^2}{\eps^2} d_\eps \right)  = f_\eps^r \EE_r + f^z_\eps \EE_z.
\end{align}
with $L^1$-bounds on $(f^r_\eps,f^z_\eps)$. Although, by the previous subsection, $(f^r_\eps)_\eps$ might concentrate at $r=0$, we have by the axisymmetry of $f^r_\eps$ 
 $$ f_\eps^r \quad \weak^* \quad    f^r + g \mathcal{H}^1\llcorner\ZZ $$
 for some $g \in L^1(\ZZ)$, but
  $$ f_\eps^r   \EE_r  \quad \weak^* \quad f^r  \EE_r   $$
  as measures. Concentration might occur but the limit measure is $0$ due to the mulitplication by $\EE_r$. Either way, the concentration in \eqref{momentumRadial} does not persist in the weak formulation of the limit.
\end{remark}

The concentration  of the right-hand side of the momentum equation in  \eqref{ginzburgLandau}  might still occur in the $\EE_z$-component   \eqref{momentumZ}. Here we employ a cancellation procedure for \eqref{momentumZ}. Hence we verify that the weak formulation of Definition \ref{defWeakSolution} is still satisfied by the limit $(u,d, P)$.

\subsection{Proof of Theorem \ref{thm:weakExistence}}
\label{subsection:ProofThm}

Eventually, we combine the results of the previous subsections to conclude Theorem \ref{thm:weakExistence}. To begin with, standard steps are needed. By the energy law \eqref{energyInequality}, we have the following result:

\begin{lemma}   \label{lemmaAprioriEstimates}
 Under the  assumptions of \eqref{wellpreparedData1} and
\eqref{wellpreparedData2}, the following a-priori estimates hold true independently of $\eps>0$:
\begin{align*}
	 \no{u_\eps}{L^2(0,T; H_0^1(\Omega))} &\leq C, \\
	 \no{u_\eps}{L^\infty(0,T; L^2_{\div} (\Omega) ) }& \leq C, \\
     \no{\partial_t u_\eps}{L^2(0,T; W^{-1,q}_{0,\div}(\Omega))} & \leq C, \\
     \no{d_\eps}{L^\infty(0,T; L^\infty(\Omega))} & \leq 1, \\
	 \no{\nabla d_\eps}{L^\infty(0,T; L^2 (\Omega) ) }&\leq C, \\
	 \no{\Delta d_\eps + \frac{1-|d_\eps|^2}{\eps^2}d_\eps}{L^2(0,T; L^2(\Omega))} &\leq C
\end{align*}
for some $1<q<\infty$.
\end{lemma}
\begin{proof}
    This follows from \eqref{energyLawGL} and the maximum principle (see \cite{kortum2020} for the details).
\end{proof}
From here and the Aubin-Lions Lemma, we deduce the existence of a subsequence with a weak limit $(u,d)$ such that 
\begin{align*}
	 (u_\eps,d_\eps)  \quad	&\weak^*  \quad (u,d)&& \text{ in } C_{\rm weak}([0,T],L^2(\Omega)), \\
	 (u_\eps,d_\eps) \quad & \to 				\quad (u,d) && \text{ in } L^2(\Omega \times (0,T)), \\
	 (u_\eps(t), d_\eps(t))  \quad & \to  \quad (u(t),d(t)) && \text{ in } L^2(\Omega) \text{ for a.e. } t \in (0,T).
\end{align*}

In what follows, we first verify the weak formulation \eqref{weakMomentum} for the limit $(u,d)$ for solenoidal  test functions $ \phi$ without swirl and with $\phi^z(0,z,t)=0$. In a second step, we perform a cut-off procedure and also show the weak identity for $\phi^z(0,z,t) \neq 0$.

By Fatou's lemma and Lemma \ref{lemmaAprioriEstimates} (see \cite{du2022,kortum2020}), the set 
$A \subset (0,T)$ has full measure such that for all $t \in A$
\begin{align*}
    \liminf_{\eps\to 0} \left(  \no{u_\eps(t)}{H_0^1(\Omega)} 
    + \no{\partial_t u_\eps(t)}{ W^{-1,q}_{0,\div}(\Omega)} +\no{\tau_\eps(t)}{L^2(\Omega)}   \right) <+\infty
\end{align*}
where $\tau_\eps =\Delta d_\eps + \frac{1-|d_\eps|^2}{\eps^2}d_\eps $.
Hence for every $t \in A$, we test $\eqref{ginzburgLandau}_1$ by $\phi(\cdot,t)$ and passing to the limit allows to deduce 
\begin{align} \label{weakRadialLimit}
	\langle \partial_t u , \phi \rangle +  \int_\Omega & \left[    (u\cdot \nabla ) u \cdot \phi(t) + \nabla u : \nabla \phi(t)   +(\nabla d)^\top \Delta d \cdot \phi(t)  \right] \d x \d t =0 
\end{align}
by the non-concentration results away from $\ZZ$ of Section \ref{subsection:noConcentrationAwayZ} and $\phi|_{r=0}=0$.  As $A$ has full measure, the weak formulation \eqref{weakMomentum} holds true by integration by parts.

 Now, we turn our attention towards removing the condition $\phi^z|_{r=0}=0$. In this case, we  employ a method  phrased concentration cancellation and follow ideas illustrated in \cite{frehse1982} and \cite{evans1990}.  
To begin with, let us highlight the existence of suitable cut-off functions in order to conclude the limit passage (see e.g.\ \cite{lopes1999}).

\begin{lemma}  \label{lemmaCutOffFunctions} There exists a family of smooth  radial functions $\eta_k: \R^2 \to  \R$ such that 
	      \begin{align*}
	   				& \eta_k \to 0  &&\text{a.e.\ and in any } L^p(\R^2), \\
	   				& \eta_k \equiv 1  &&\text{on } B_{1/k}(0) \text{ and} \\
	   				& \nabla \eta_k \to 0 && \text{in } L^2(\R^2) 
	   \end{align*}
	   as $k$ tends to $\infty$.
\end{lemma}
\begin{proof} This follows directly from the fact that isolated points have vanishing 2-capacity in $\R^2$ (c.f.\ \cite{evans2015}).  However, one might be more explicit and consider
$\eta_k:\R^2 \to \R$  defined by 
 $$ \eta_k(r) = \begin{cases} 
	   		1, & \text{if } r \leq \frac1k, \\
	   		\frac{-\log \left( \sqrt{k}r\right)}{\log \sqrt{k}}, & \text{if } \frac1k \leq r \leq \frac{1}{\sqrt{k}}, \\
	   		0, & \text{if } r \geq \frac{1}{\sqrt{k}}.
	   \end{cases}$$
Convolution with a standard convolution kernel of small enough $k$-dependent support then yields the result.
\end{proof}
Eventually, we show that the weak formulation of \eqref{momentumZ} is satisfied in the limit. In order to do so, recall that the effective test  functions are given of the form 
\begin{align*}
  \Phi =     - \frac{\pa_z \psi}{r} \EE_r + \frac{\pa_r  \psi}{r} \EE_z
\end{align*}
for smooth scalar $\psi=\psi(r,z,t)$. In order to perform the cut-off, we take stream functions 
$$ \psi_k = \int_0^r (1-\eta_k) \partial_r \psi \d s $$
which leads to 
\begin{align*}
    \Phi_k (r,z,t)=  -  \frac{\pa_z \psi_k}{r} \EE_r +  (1- \eta_k) \frac{\pa_r  \psi}{r} \EE_z
\end{align*}
By Lemma \ref{lemmaCutOffFunctions} and the independence in time of $\eta_k$, one easily checks that 
\begin{align}
    \Phi_k \quad &\to \quad \Phi   & &\text{ in } C([0,T], H^1_0(\Omega)), \notag \\
   \partial_t \Phi_k  \quad &\to \quad \partial_t \Phi   &&\text{ in } C([0,T], H^1_0(\Omega)), \label{convergenceTestFunction} \\
    \frac{\partial_z \psi_k}{r}  \quad &\to  \quad  \frac{\partial_z \psi}{r}&& \text{ uniformly in a neighborhood around } \ZZ . \notag
\end{align}
 Thus, using $\Phi_k$ as legitimate test function in \eqref{defWeakSolution}, we have
\begin{align*}
&\int_0^\infty \int_\Omega\left[ -  u \cdot \partial_t \Phi_k +( u \cdot \nabla u)\cdot  \Phi_k+ \nabla u \cdot \nabla \Phi_k  \right]  \d x \d t  \\
& =- \int_0^\infty \int_\Omega \left( (\nabla d)^\top \Delta d\right) \cdot    \Phi_k  \d x \d t +\int_\Omega u_0 \cdot  \Phi_k(\cdot, 0) \d x
\end{align*}
 for every $k \in \N$. This is however only useful in case we are capable of letting $k \to \infty$ in an appropriate way. Thus, we  check convergence properties of all terms to conclude the weak formulation. Each term is considered separately:

By Lemma \ref{lemmaAprioriEstimates}, we have $(\nabla d)^\top \Delta d \in L^2_tL^1_x( \Omega\times (0,\infty))$.  With respect to \eqref{convergenceTestFunction}, we exploit the Dominated Convergence Theorem (in particular for $\Phi_k^z$) which yields
		  $$ \int_0^\infty \int_\Omega \left( (\nabla d)^\top \Delta d \right)\cdot    \Phi_k    \quad \to \quad  \int_0^\infty \int_\Omega  \left( (\nabla d)^\top \Delta d \right)\cdot   \Phi.$$
Next, it holds  
	   $$  \int_0^\infty \int_\Omega  u \cdot \partial_t \Phi_k  \quad \to \quad  \int_0^\infty \int_\Omega  u \cdot  \partial_t \Phi$$
	   by Lemma \ref{lemmaAprioriEstimates} and $u \in L^2(\Omega \times (0,T))$. Again by the Dominated Convergence Theorem, we have 
	    $$  \int_0^\infty \int_\Omega  (u\cdot \nabla u) \cdot \Phi_k  \quad \to \quad  \int_0^\infty \int_\Omega  (u \cdot \nabla u) \cdot  \Phi.$$
 Due to \eqref{convergenceTestFunction}, $\nabla u \in L^2_{t,x} ( \Omega \times (0,\infty))$, and Lemma \ref{lemmaCutOffFunctions}, it holds 
	   $$  \int_0^\infty \int_\Omega \nabla u : \nabla \Phi_k  \quad  {\rightarrow} \quad \int_0^\infty \int_\Omega \nabla u : \nabla \Phi$$
and finally  $\displaystyle\int_\Omega  u_0 \cdot  \Phi_k(\cdot,0) ~ \to ~\int_\Omega u_0 \cdot \Phi(\cdot,0).$

In total, it holds
\begin{align*}
\int_0^\infty \int_\Omega& \left[ -  u \cdot  \partial_t \phi +(u \cdot \nabla) u\cdot \phi + \nabla u : \nabla \phi \right]  \d x \d t  \\
& = -\int_0^\infty \int_\Omega (\nabla d)^\top \Delta d \cdot \phi \d x \d t+\int_\Omega u_0 \cdot \phi(\cdot,0) \d x
\end{align*}
for every smooth solenoidal $\phi =\phi^r \EE_r +\phi^z \EE_z$, i.e.\ the weak formulation \eqref{weakMomentum}  is valid. \\
Regarding \eqref{weakDirector}, we sketch the argument which relies on an argument by \cite{chen1989} (see e.g.\ \cite{kortum2020} for details). Taking the vector product of the director equation of \eqref{ginzburgLandau} by $d_\eps$, we have
 \begin{align*}
      d_\eps \times (\partial_t d_\eps + (u_\eps \cdot \nabla ) d_\eps  - \Delta d_\eps) =0.
 \end{align*}
Realizing that $ d_\eps \times \Delta d_\eps  = \div ( d_\eps \times \nabla d_\eps)$ holds true, we pass to the limit by standard arguments. Taking once more the cross product of the resulting equation by $d$ and using $d\times d \times \Delta d = - \Delta d - |\nabla d |^2 d$, since $|d|=1$, we obtain \eqref{weakDirector}. \\
The energy inequality \eqref{energyInequality} is a consequence of \eqref{energyLawGL} and the strong convergence of initial data. This concludes the proof of Theorem \ref{thm:weakExistence}.

\section{Weak compactness of solutions of axisymmetric EL system}

This section is devoted to the proof of Theorem \ref{thm:weakCompactness} on the weak compactness of axisymmetric without swirl solutions of the Ericksen-Leslie system
\eqref{momentum}--\eqref{director}.

First it follows from the Definition 1.1 that $(u_n, d_n)$ satisfies the energy inequality \eqref{energyInequality} so that 
\begin{align}\label{energyInequality1}
    \sup_{0\le t\le T}\int_\Omega e(u_n,d_n)(x,t)\,\d x&+
    \int_{Q_T}(|\nabla u_n|^2+
    |\Delta d_n+|\nabla d_n|^2 d_n|^2\,\d x \d t\nonumber\\
    &\le \int_\Omega e(u_{0n},d_{0n})\,\d x\le C(u_0,d_0),
\end{align}
where $\displaystyle e(u,d)=\int_\Omega (|u|^2+|\nabla d|^2)\,dx$ denotes the total energy of $(u,d)$, and $C(u_0, d_0)$ denotes a positive constant depending on the total energy of $(u_0, d_0)$.

Write $u_n(r,\theta,z,t)=u_n^r(r,z,t)\EE_r
+u^z(r,z,t)\EE_z$ and
$d_n(r,\theta, z, t)
=\sin\phi_n (r,z,t)\EE_r+\cos\phi_n(r,z,t)\EE_z$.
It follows from equation \eqref{director} that $\phi_n$ satisfies
\begin{equation}\label{phiEquation}
\partial_t\phi_n +b_n\cdot\nabla\phi_n-L\phi_n+\frac{\sin 2\phi_n}{2r^2}=0,    
\end{equation}
where $b_n=u^r_n\EE_r+u^z_n\EE_z$
and $L=\partial^2_r+r^{-1}\partial_r+\partial^2_z$. 

From \eqref{energyInequality1} and \eqref{phiEquation}, one can see that $\nabla^2\phi_n\in L^2((0,T),
\Omega_\delta)$ for any $\delta>0$,
where
$\Omega_\delta=\big\{(r,\theta,z)\in\Omega:
r>\delta\big\}$. This implies that for
a.e. $t\in (0,T)$, $d_n(\cdot, t)$ is
a {\it suitable} approximated harmonic map (see \cite[Definition 5.1]{lin2016}) on $\Omega_+=\Omega\cap\big\{r>0\big\}$ with
$L^2$-tension field $\tau_n=
(\partial_t d_n+u_n\cdot\nabla d_n)(\cdot, t)$ (see \cite[Remark 5.2]{lin2016}). In particular, for any $\Lambda>0$, if we define
$$
\mathcal{H}_\Lambda^T=\Big\{t_0\in (0, T):\ \Lambda(t_0)=\liminf_{n\to \infty}
\int_\Omega |\tau_n(x, t_0)|^2\,dx\le \Lambda\Big\},
$$
then 
$$\mathcal{L}^1((0,T)\setminus \mathcal{H}_\Lambda^T)\le \frac{E_0}{\Lambda}.$$
Similar to the Section \ref{section:GL}, by applying \cite[Lemma 5.3, Lemma 5.4]{lin2016}, and the first part of proof of \cite[Theorem 7.1]{lin2016} we have that for any $t_0\in\mathcal{H}_\Lambda^T$ and  $p_0=(r_0,\theta_0, z_0)\in\Omega_+$, if $0<r\le R<\min\big\{r_0,
    {\rm{dist}}(p_0, \partial\Omega_+)\big\}$, then 
    \begin{equation}
     \Psi_r(d_n(t_0), p_0)
     \le \Psi_R(d_n(t_0), p_0),
    \end{equation}
    where
    $$
    \Psi_r(d_n(t_0), p_0)
    =\frac{1}{r}\int_{B_r(p_0)} \left(\frac12|\nabla d_n(t_0)|^2-\langle (p-p_0)\nabla d_n(t_0), \tau_n\rangle\right)+\frac12\int_{B_r(p_0)}|p-p_0||\tau_n|^2.
    $$
    In particular, it holds that
    \begin{equation}\label{almost_monotonicity}
     r^{-1}\int_{B_r(p_0)} |\nabla d_n(t_0)|^2
     \le 2R^{-1}\int_{B_R(p_0)}|\nabla d_n(t_0)|^2
     +4\Lambda R.
    \end{equation}
If we define the energy concentration set
of $d_n(t_0)$ in $\Omega_+$ by
\[
\Sigma_{t_0}^+
=\bigcap_{0<r<{\rm{dist}}(p_0,\partial\Omega_+)}\Big\{
p_0\in\Omega_+: \ \liminf_{n\to\infty}
r^{-1}\int_{B_r(p_0)} |\nabla d_n(t_0)|^2>\varepsilon_0^2\Big\},
\]
then $\Sigma_{t_0}^+$ is a closed set with
$\mathcal{H}^1(\Sigma_{t_0}^+\cap K)<\infty$
for any compact set $K\subset\mathbb R^3$,
and there exists an approximated harmonic map
$d:\Omega\to\mathbb S^2$ with tension field
$\tau=\displaystyle\lim_{n\to\infty} \tau_n$ weakly in $L^2(\Omega)$ such that, after passing to a subsequence, $d_n(\cdot,t_0)\rightharpoonup d$
in $H^1(\Omega)$, and
$$d_n(\cdot, t_0)\rightarrow d 
\ \ {\rm{in}}\ \ H^1_{\loc}(\Omega_+\setminus\Sigma_{t_0}^+),$$
and
$$|\nabla d_n(t_0)|^2\,\d x\rightharpoonup |\nabla d(t_0)|^2\, \d x$$
as convergence of Radon measures on $\Omega_+\setminus\Sigma_{t_0}^+$. 

Next we need
\begin{lemma}\label{nonconcentration}
Under the same notations as above. For any $t_0\in \mathcal{H}_\Lambda^T$, $\Sigma_{t_0}^+=\emptyset$.
\end{lemma}
\begin{proof}

Write $d_n(t)=\sin\phi_n(r,z,t)\EE_r +\cos\phi_n(r,z,t)\EE_z$,
and $u_n(t)=u^r_n(r,z,t)\EE_r+u_n^z(r,z,t)\EE_z$. Denote
by $\phi_n(r,z)=\phi_n(r,z,t_0)$ and $u_n(r,z)=u_n(r,z,t_0)$.
Then we have
\begin{equation}\label{phiEqn}
L\phi_n-\frac{\sin 2\phi_n}{2r^2}=\widehat{\tau}_n:=\partial_t \phi_n(\cdot, t_0)+u_n^r \partial_r \phi_n+u^z_n\partial_z\phi_n 
\ \ \ {\rm{in}}\ \ \ \mathbb D_+,
\end{equation}
where $\mathbb D_+=\big\{(r,z): \ (r,0,z)\in\Omega_+\big\},$
and 
$$\|\widehat{\tau}_n\|_{L^2(\mathbb D_+)}
=\frac1{2\pi}\|\tau_n(t_0)\|_{L^2(\Omega_+)}\le \Lambda^\frac12.$$

Similar to Section \ref{section:GL} above, for simplicity we may assume that
$\Sigma_{t_0}=\big\{(r_0, 0, z_0)\big\}\times \mathbb S^1$ for a point $(r_0,z_0)\in\mathbb D_+.$ Recall the energy density of
$d_n$ is given by
$$
\frac12|\nabla d_n|^2=e(\phi_n):=\frac12\left(|\nabla_{r,z}\phi_n|^2+\frac{\sin^2\phi_n}{r^2}\right).
$$
Hence by the definition of $\Sigma_{t_0}^+$ we have 
\begin{align}
\eps_0^2&<\liminf_{n\to\infty} 
\frac{1}{r}\int_{\mathcal{C}^3(r_0,0,z_0)} e(\phi_n) r\d r\d \theta dz\nonumber\\
&=\liminf_{n\to\infty} 
\int_{\mathcal{C}^2(r_0,0,z_0)} e(\phi_n) r\d r\d z,
\quad  0<r<\frac{r_0}{16}. 
\end{align}

As in \eqref{max_exhau} and \eqref{max_exhau1}, if we define the maximum concentration function of $\phi_n$ by
\begin{align}\label{max_exhau4}
    \Theta_n((r_0,z_0);\lambda)
    =\max\Big\{\int_{\mathcal{C}^2_\lambda(r_*,z_*)} e(\phi_n)(r,z)r\d r\d z: (r_*,z_*)\in\mathcal{C}^2_{\frac{r_0}{32}}(r_0, z_0)\Big\},
\end{align}
then for a sufficiently large $C_*>0$ there are $\lambda_n\in (0,\frac{r_0}{32})$ and $(r_n, z_n)\in \mathcal{C}^2_{\frac{r_0}{32}}(r_0, z_0)$ such that
\begin{align}\label{max_exhau5}
\frac{\varepsilon_0^2}{C_*}=\Theta_n((r_n,z_n), \lambda_n)
=\int_{\mathcal{C}^2_{\lambda_n}(r_n,z_n)}e(\phi_n)(r,z)r\d r\d z.
\end{align}
Arguing as in Lemma 2.2, it holds that
\begin{align}\label{blowScale}
\lim_{n\to\infty}\lambda_n=0
\ \ {\rm{and}}\ \
\lim_{n\to\infty}(r_n,z_n)=(r_0,z_0).
\end{align}
Define the blowup sequence 
$$\psi_n(r,z)=\phi_n(r_n+\lambda_n r,
z_n+\lambda_n z),  \ \forall (r,z)\in \mathbb{D}_n:=\lambda_n^{-1}\big(\mathbb D_+\setminus\{(r_n,z_n)\}\big).$$ 
It follows from equation \eqref{phiEqn} that $\psi_n$ solves
\begin{align}\label{psiEqn}
    L_n\psi_n=\frac{\lambda_n^2\sin 2\psi_n}{2(r_n+\lambda_n r)^2}
    +\widetilde{\tau}_n \ \ {\rm{in}}
\ \ \mathbb{D}_n,
\end{align}
where
$$
L_n=\partial^2_r+\frac{\lambda_n}{r_n+\lambda_n r} \partial_r+\partial^2_z,
$$
and
$$
\widetilde{\tau}_n(r,z)
=\lambda_n^2\widehat{\tau}_n(r_n+\lambda_n r,
z_n+\lambda_n z).
$$
Furthermore, \eqref{max_exhau5} yields
\begin{align}\label{max_exhau7}
&\frac{\varepsilon_0^2}{C_*}
=\int_{\mathcal{C}^2_{1}(0,0)}e(\psi_n)(r,z)(r_n+\lambda_n r)\d r\d z\nonumber\\
&\ge \int_{\mathcal{C}^2_{1}(r_*,z_*)}e(\psi_n)(r,z)(r_n+\lambda_n r)\d r\d z, \ \forall (r_*,z_*)\in \mathbb D_n.
\end{align}
Here
$$
e(\psi_n)(r,z)=\frac12\left(|\nabla_{r,z}\psi_n|^2+\frac{\lambda_n^2}{(r_n+\lambda_n r)^2}\sin^2\psi_n\right)(r,z)
$$
On the other hand, similar to \eqref{total_energy}, we can derive
\begin{align}\label{total_energy1}
\int_{\mathcal{C}^2_R(0,0)} e(\psi_n)(r,z)(r_n+\lambda_n r)\,\d r\d z
\le C(E_0+\Lambda), \ \ \forall R>0.
\end{align}
From \eqref{blowScale}, \eqref{psiEqn}, and \eqref{max_exhau7}, we can write
\begin{equation}\label{psinEqn}
(\partial_r^2+\partial^2_z)\psi_n 
=f_n:=\frac{\lambda_n^2\sin 2\psi_n}{2(r_n+\lambda_n r)^2}
    +\widetilde{\tau}_n-\frac{\lambda_n}{r_n+\lambda_n r}\partial_r \psi_n \ \ {\rm{in}}\ \ \mathbb D_n,
\end{equation}
with
$$
\int_{B_R^2}e(\psi_n)\,\d r\d z\le C(R, r_0),
\ \ {\rm{and}}\ \ \big\|f_n\big\|_{L^2(B_R^2)}\le C(R,r_0)\lambda_n, \ \forall R>0.
$$
Hence by $W^{2,2}$-estimates (see \cite{gilbarg2001}), one has
$$
\int_{B_R^2} |\nabla^2_{r,z}\psi_n|^2\,\d r\d z\le C(R,r_0),\ \forall R>0.
$$
Thus we may assume that there exists $\psi\in H^{1}_{\rm{loc}}(\mathbb R^2)$ such that $\psi_n\rightarrow \psi$ in $H^1_{\rm{loc}}(\mathbb R^2)$. Passing to the limit in equation \eqref{psinEqn} and condition \eqref{max_exhau7}, we have that
\begin{equation}\label{harmonic}
(\partial_r^2+\partial^2_z)\psi=0 \ \ {\rm{in}}\ \ \R^2,
\end{equation}
and 
\begin{equation}\label{nontrivial}
\int_{\mathcal{C}^2_1(0,0)}(|\partial_r\psi|^2+|\partial_z\psi|^2)\,drdz=
\frac{\eps_0^2}{r_0C_*}>0.
\end{equation}
On the other hand, by \eqref{total_energy1} we have that
\begin{equation}\label{total_energy2}
\int_{\R^2}(|\partial_r\psi|^2+|\partial_z\psi|^2)\,drdz\le \frac{C(E_0+\Lambda)}{r_0}<\infty.
\end{equation}
It is readily seen that \eqref{harmonic} and \eqref{total_energy2} imply
that $\psi$ must be a constant, which contradicts to \eqref{nontrivial}.
This proves that $\Sigma_{t_0}^+=\emptyset$.
\end{proof}

Now we can complete the proof of Theorem \ref{thm:weakCompactness}.
\begin{proof} First, for $\mathcal{L}^1$ a.e. $t_0\in (0, T)$, we have that $\Lambda(t_0)<\infty$ so that Lemma \ref{nonconcentration} yields that $\Sigma_{t}^+=\emptyset$ for a.e. $t\in (0,T)$. With this at hand, we can first argue exactly as in the proof of Theorem \ref{thm:weakExistence} given in Section \ref{subsection:noConcentrationAwayZ} to conclude that the $L^2_tL^2_x\times L^2_tH^1_x$-weak limit map $(u, d):\Omega\times (0,T)\to \mathbb R^3\times \mathbb S^2$, of $(u_n, d_n)$, satisfies equation \eqref{weakMomentum} for all test functions $\phi\in C^\infty_{0,{\rm{div}}}(\Omega_+\times [0,T))$. Then we can apply the concentration-cancellation property given by Section \ref{subsection:ProofThm} to show that $(u, d)$ satisfies equation \eqref{weakMomentum} for all test functions $\phi\in C^\infty_{0,{\rm{div}}}(\Omega\times [0,T))$. While it is not hard to verify that $(u,d,P)$
also satisfies equation \eqref{weakDirector}, the initial and boundary value
$(u_0, d_0)$, as well the energy inequality \eqref{energyInequality}. Hence by Definition \ref{defWeakSolution} $(u,d,P)$ is a weak solution of the equations \eqref{momentum}--\eqref{director}
and the initial-boundary condition \eqref{iniData}-\eqref{boundaryData}. Again, since the property of axisymmetric without swirl is preserved under this convergence, we conclude
that $(u, d,P)$ is also axisymmetric without swirl.
\end{proof}

\section*{Appendix}

In this section, we will sketch the construction of an axisymmetric with no swirl weak solution
to \eqref{ginzburgLandau}, 
\eqref{initialDataGL} and \eqref{boundaryDataGL}.
The argument is based on suitable extensions to the axisymmetric setting of the Galerkin method outlined by \cite{lin1995}. 

Let us first introduce two function spaces on
an axisymmetric bounded smooth domain $\Omega\subset\R^3$. Namely,
\[
\mathbb H_{\rm{sym}} 
=\Big\{\ v\in L^2(\Omega,\R^3):
\ v \mbox{ is divergence-free and axisymmetric without swirl }, v\cdot n \big|_{\partial \Omega} =0\Big\},
\]
and
\[
\mathbb V_{\rm{sym}} 
=\Big\{\ v\in H^1_0(\Omega,\R^3):
\ v \mbox{ is divergence-free and axisymmetric without swirl }\Big\}.
\]
Observe that if we denote
\[
\mathbb X=\Big\{v\in C_0^\infty(\Omega,\R^3):
\ v \ \mbox{is divergence-free and
axisymmetric without swirl} \Big\},
\]
then $\mathbb H_{\rm{sym}}$ is the closure of
$\mathbb X$ in $L^2(\Omega)$,
and $\mathbb V_{\rm{sym}}$ is the closure of
$\mathbb X$ in $H^1_0(\Omega)$.
Denote the dual space of $\mathbb V_{\rm{sym}}$
by $\mathbb V_{\rm{sym}}'$.

We begin with the following Lemma on the orthonormal basis of $\mathbb{H}_{\rm{sym}}$ by
eigenfunctions of the Stokes operator.

\begin{lemma} There exists an orthonormal basis 
of $\mathbb{H}_{\rm{sym}}$ consisting of
eigenfunctions $\phi_i\in C^\infty(\Omega,\R^3)\cap \mathbb{V}_{\rm{sym}}$, $i=1, 2,...$, of the Stokes equation on $\Omega$:
\[
\begin{cases}
-\Delta \phi_i+\nabla P_i=\lambda_i \phi_i
\ &{\rm{in}}\ \Omega,\\
\nabla\cdot\phi_i=0 \ &{\rm{in}}\ \Omega,\\
\phi_i=0 \ &{\rm{on}}\ \partial\Omega.
\end{cases}
\]
Here $0<\lambda_1\le \lambda_2\le\cdots\lambda_n\le\cdots$, with $\displaystyle\lim_{n\to\infty}\lambda_n=\infty$.
\end{lemma}
\pf Define the Stokes operator 
$S: \mathbb{V}_{\rm{sym}}\to \mathbb{V}_{\rm{sym}}'$ by letting
$$
\langle Sv, w\rangle=\int_\Omega \nabla v\cdot\nabla w\,dx, \ \forall v, w\in \mathbb{V}_{\rm{sym}}.
$$
One can apply the Lax-Milgram theorem to conclude
that $S$ is a bijective, bounded linear map, and $S^{-1}:
\mathbb{V}_{\rm{sym}}'\to \mathbb{V}_{\rm{sym}}$
is a compact linear map. Hence the conclusion follows. \qed

Now we can set up the modified Galerkin method as in \cite[pages 506-507]{lin1995}. More precisely, 
for $m\ge 1$, let $\mathbb{P}_m:\mathbb H_{\rm{sym}}\to \mathbb H^m_{\rm{sym}}
\equiv {\rm{span}}\big\{\phi_1,\cdots,\phi_m\big\}$ denote the orthogonal projection. For any fixed $\eps>0$ and $0<T<\infty$, denote $Q_T=\Omega\times (0,T)$ and 
consider the following problem: 
\begin{align}
&\partial_t v_m
=\mathbb{P}_m \Big(\Delta v_m-v_m\cdot\nabla v_m
-\nabla\cdot(\nabla d_m\otimes\nabla d_m)\Big)
\ {\rm{in}}\ Q_T,\\
&v_m(\cdot, t)\in \mathbb{H}^m_{\rm{sym}}\ \ \ 0<t<T,\\
&\partial_t d_m+v_m\cdot\nabla d_m
=\Delta d_m +\frac{1}{\eps^2}(1-|d_m|^2)d_m
 \ {\rm{in}}\ Q_T,\\
&v_m(\cdot, 0)=\mathbb{P}_m v_0 \ \ \ 
 \ {\rm{in}} \ \ \ \ \  \Omega,\\
&d_m(\cdot, 0)=d_0 \ \ \ \ \ \ \ \ {\rm{in}} \ \ \ \ \  \Omega,\\
&d_m(\cdot, t)=d_0\ \ \ \ \ \ \ \ {\rm{on}}\ \ \ \ \ \partial\Omega\times [0,T].
\end{align}
By the theory of parabolic equations, we know that for any given $v_m(\cdot, t)\in \mathbb H^m_{\rm{sym}}$, the initial boundary problem
(4.3), (4.5), and (4.6) has a unique strong solution $d_m: Q_T\to \mathbb R^3$, that is axisymmetric without swirl. In particular
$d_m(\cdot, t)\in \mathbb V_{\rm{sym}}$ for all
$0\le t<T$. 

Next we can follow the proof of \cite[Theorem 2.1]{lin1995} line by line to show 
\begin{theorem} \label{Galerkin} For all $m\ge 1$,
$0<T<\infty$, and $\eps>0$, if $v_0\in \mathbb{H}_{\rm{sym}}$ and $d_0\in H^1(\Omega,\mathbb S^2)$, with $d_0\in H^{\frac32}(\partial\Omega,\mathbb S^2)$, is axisymmetric without swirl, the problem (4.1)--(4.6) admits a unique classical solution 
$v_m(\cdot, t)\in \mathbb{H}_{\rm{sym}}^m$, and $d_m(\cdot, t)$ that is axisymmetric without swirl for $0\le t<T$. Furthermore, $(v_m, d_m)$ satisfies
the followng energy inequality:
\begin{align}\label{globalenergyineq}
\sup_{0\le t\le T}
\int_\Omega \big(|v_m|^2+|\nabla d_m|^2+\frac{1}{2\eps^2}(1-|d_m|^2)^2\big)
&+2\int_{Q_T}\big(|\nabla v_m|^2+|\Delta 
d_m+\frac{1}{\epsilon^2}(1-|d_m|^2)d_m|^2\big)
\nonumber\\
&\le 
\int_\Omega \big(|v_0|^2+|\nabla d_0|^2\big).
\end{align}
\end{theorem}

\medskip
Finally, as in \cite[pages 512-512]{lin1995}, by \eqref{globalenergyineq} we can pass the limit $m\to\infty$ of
the system (4.1)--(4.6) to obtain a weak solution
$(u_\eps, d_\eps)$ of \eqref{ginzburgLandau}, \eqref{initialDataGL} and \eqref{boundaryDataGL},
that is axisymmetric without swirl in $Q_T$,  and
$(u_\eps, d_\eps)$ satisfies \eqref{energyLawGL}.

\bigskip
\bigskip
\noindent{\bf Acknowledgments:}  The second author is partially supported by NSF grant 2101224.
The first author would like to thank Francesco De Anna for helpful discussions on the subject.


\begin{thebibliography}{19}
\baselineskip=12pt
{\footnotesize



% \bibitem{ball2017}
% Ball, J.M.: Mathematics and liquid crystals. {\it Mol. Cry. Liq. Cry.} {\bf 647}(1) (2017), 1--27.

%\bibitem{bertsch2002}
%Bertsch, M., Dal Passo, R., van der Hout, R.: Nonuniqueness for the %Heat Flow of Harmonic Maps on the Disk. {\it Arch. Ration. Mech. %Anal.} {\bf 161} (2002), 93--112.

\bibitem{bethuel1991}
Bethuel, F.: The approximation problem for Sobolev maps between two manifolds. {\it Acta Math.} {\bf 167} (1991), 153--206.

\bibitem{bethuel1994}
Bethuel, F., Brezis, H., Helein, F.: {\it Ginzburg-Landau vortices}, 
Birkhäuser, Boston, 1994.

\bibitem{chang1992}
Chang, K. C., Ding, W. Y., Ye, R. G.: Finite-time blow-up of the heat flow of harmonic maps from surfaces. {\it J. Diff. Geom.} {\bf36} (1992), 507--515.

\bibitem{chen1989}
 Chen, Y. M.: Weak solutions to the evolution problem for harmonic maps into spheres. {\it  Math. Z.} {\bf 201} (1989), 69--74.

\bibitem{chen1989b}
 Chen, Y. M., Struwe, M.: Existence and partial regularity results for the heat flow of harmonic maps. {\it Math. Z.} {\bf 201} (1989), 83--103.
 
  \bibitem{deanna2019}
De Anna, F., Liu, C.: Non-isothermal General Ericksen–Leslie System: Derivation, Analysis and Thermodynamic Consistency. {\it Arch. Ration. Mech. Anal.}  {\bf 231} (2019), 637--717.
 

\bibitem{du2022}
Du, H. R., Huang, T., Wang, C. Y.: Weak compactness property of simplified nematic liquid crystal flows in dimension two. {\it Math. Z.} {\bf 302} (2022), 2111--2130.
 
 \bibitem{delort1991}
Delort, J. M.: Existence de nappes de tourbillon en dimension deux. {\it J. Amer. Math. Soc.} {\bf 4} (1991),
553--586.
 
 \bibitem{diperna1988}
DiPerna, R., Majda, A.: Reduced Hausdorff dimension and concentration-cancellation for two-dimensional incompressible flow. {\it J. Amer. Math. Soc.} {\bf 1} (1988), 59--95.
 
\bibitem{ericksen1962}
Ericksen, J. L.: Hydrostatic theory of liquid crystals. {\it Arch. Ration. Mech. Anal.} {\bf 9} (1962), 371--378.

\bibitem{evans1990}
Evans, L. C.: {\it Weak Convergence Methods for Nonlinear Partial Differential Equations}, 
AMS, 1990.

\bibitem{evans2015}
Evans, L. C., Gariepy, R. F.: {\it Measure Theory and Fine Properties of Functions}, 
Taylor \& Francis, Textbooks in Mathematics, 2nd edition, 2015.
 
\bibitem{feng2020}
Feng, Z., Hong, M. C., Mei, Y.: {\it Convergence of the Ginzburg-Landau approximation for the Ericksen-Leslie system.} {\it SIAM J. Math. Anal}. \textbf{52} (2020), no. 1, 481–-523.

 
\bibitem{frehse1982}
Frehse, J.: Capacity methods in the theory of partial differential equations. {\it Jber. Deutsch. Math. Verein,} \textbf{84} (1982), 1--44. 
 
\bibitem{gilbarg2001}
Gilbarg, D., Trudinger, N.S.: {\it Elliptic Partial Differential Equations of Second Order}, 2nd.~ed., Springer Berlin Heidelberg, 2001.

%\bibitem{grisvard2011}
%Grisvard, P.: {\it Elliptic Problems in Nonsmooth Domains}, SIAM, %2011.

%\bibitem{harpes2004}
%Harpes, P.: Partial compactness for the 2-D Landau-Lifshitz flow. %{\it Electron. J. Differential Equations} {\bf 90} (2004), 1--24.

%\bibitem{helein1990}
%H{\'e}lein, F.: R{\'e}gularit{\'e} des applications faiblement %harmoniques entre une surface et une sphere. {\it Comptes rendus de %l'Acad{\'e}mie des sciences. S{\'e}rie 1, Math{\'e}matique} {\bf 311}%(9) (1990), 519--524.

\bibitem{hieber2017} Hieber, M. G., Pr\"uss, J. W.: 
Dynamics of the Ericksen–Leslie equations with general Leslie stress I: the incompressible isotropic case. {\it Math. Ann.} {\bf 369} (2017), 977--996.

\bibitem{hieber2018}
Hieber, M. G., Pr\"uss, J. W.: Modeling and analysis of the Ericksen-Leslie equations for nematic liquid crystal flows.
{\it Handbook of Mathematical Analysis in Mechanics of Viscous Fluids}, Springer International Publishing 2018, 
1075--1134.

\bibitem{hong2010}
Hong, M. C.: Global existence of solutions of the simplified Ericksen--Leslie system in dimension two.
{\it Calc. Var. PDE} {\bf 40} (2010), 15--36.

\bibitem{huang2014}
Huang, J. R., Lin, F. H., Wang, C. Y.: Regularity and existence of global solutions to the Ericksen–Leslie system in $\R^2$. {\it Comm. Math. Phys.} {\bf331} (2014), 805--850. 

\bibitem{huang2016}
 Huang, T., Lin, F. H., Liu, C., Wang, C. Y.: Finite Time Singularity of the Nematic Liquid Crystal Flow in Dimension Three. {\it Arch. Rational. Mech. Anal. } {\bf 221} (2016), 1223–1254.

\bibitem{kortum2020}
Kortum, J.:  Concentration-cancellation in the Ericksen-Leslie model.
{ \it Calc. Var. PDE} {\bf 59}(6) (2020), no. 6, Paper No. 189, 16pp.

\bibitem{lai2022}
Lai, C. C., Lin, F., Wang, C., Wei, J., Zhou, Y.: Finite time blowup for the nematic liquid crystal flow in dimension two. {\it Comm. Pure Appl. Math.} {\bf 75}(1) (2022), 128--196.

%\bibitem{lin2016}
%Huang, T., Lin, F.H., Liu, C., Wang, C.: Finite time singularity of %the nematic liquid crystal flow in dimension three. 
%{\it Arch. Ration. Mech. Anal.} {\bf 221} (2016), 1223--1254.

\bibitem{leslie1968}
Leslie, F. M.: Some constitutive equations for liquid crystals. {\it Arch. Ration. Mech. Anal.} {\bf 28} (1968), 265--283.

\bibitem{lin2010}
Lin, F. H., Lin, J. Y., Wang, C. Y.: Liquid crystal flows in two dimensions. {\it Arch. Ration. Mech. Anal.} {\bf 197} (2010), 297--336.

\bibitem{lin2010b}
Lin, F. H., Wang, C. Y.: On the uniqueness of heat flow of harmonic maps and hydrodynamic flow of nematic liquid crystals. {\it Chin. Ann. Math. B}  {\bf 31} (2010), 921--938. 

\bibitem{lin1989}
Lin, F. H.: Nonlinear theory of defects in nematic liquid crystal: phase transition and
flow phenomena. {\it Comm. Pure Appl. Math.} {\bf 42} (1989), 789--814.

\bibitem{lin1995}
Lin, F. H., Liu, C.: Nonparabolic dissipative systems modeling the flow of liquid crystals. {\it Comm. Pure Appl. Math.} {\bf 48} (1995), 501--537.

\bibitem{lin1999}
Lin, F.H., Wang, C. Y.: Harmonic and quasi-harmonic spheres. {\it Comm. Anal. Geom.} {\bf 7}(2) (1999), 397--429.

\bibitem{lin2008}
Lin, F. H.,  Wang, C. Y.:  {\it The analysis of harmonic maps and their heat flows}, World Scientific, 2008.
 
\bibitem{lin2016}
Lin, F. H., Wang, C. Y.: Global existence of weak solutions of the nematic liquid crystal flow in dimension three. {\it Comm. Pure Appl. Math.} {\bf 69} (2016), 1532--1571.

\bibitem{liu2009}
Liu, J. G., Wang, W. C.: Characterization and regularity for axisymmetric solenoidal vector fields with application to Navier--Stokes equation. {\it SIAM J. Math. Anal.} {\bf 41}(5) (2009), 1825--1850.

\bibitem{lopes1999}
Lopes Filho, M.C., Nussenzveig Lopes, H.J., Zheng, Y.: Weak solutions for the equation of incompressible and inviscid fluid dynamics. Notes (1999).

\bibitem{majda2002}
Majda, A. J., Bertozzi, A.L.: {\it Vorticity and incompressible flow},
Cambridge Texts in Applied Mathematics, Cambridge University Press, 2002.

%\bibitem{qing1995}
%Qing, J.: On singularities of the heat flow for harmonic maps from %surfaces into spheres. {\it Comm. Anal. Geom.} {\bf 3} (1995), 297-%-315.

\bibitem{robinson2016}
Robinson, J. C., Rodrigo, J.L, Sadowski, W.: {\it The Three-Dimensional Navier--Stokes Equations}, Cambridge University Press, 2016.

%\bibitem{roubicek2013}
%Roub{\'\i}{\v{c}}ek, T.: {\it Nonlinear Partial Differential %Equations with Applications},
%International Series of Numerical Mathematics, Springer Basel, 2013.

% \bibitem{schmitt1995}
% Schmitt, B.: The poloidal-toroidal representation of solenoidal fields in spherical domains. {\it Analysis} {\bf 15}(3) (1995), 257--278.

\bibitem{schmitt1992}
Schmitt, B., von Wahl, W.:  Decomposition of solenoidal fields into poloidal fields, toroidal fields and the mean flow. Applications to the Boussinesq-equations. {\it The Navier-Stokes Equations II --- Theory and Numerical Methods} Springer Berlin Heidelberg (1992), 291--305. 

%\bibitem{schochet1995}
%Schochet, S.: The weak vorticity formulation of the 2D Euler equations and
%concentration--cancellation. {\it Comm. Partial Differential Equations} {\bf 20} (1995), 1077--1104.

\bibitem{struwe1985}
Struwe, M.: On the evolution of harmonic mappings of Riemannian surfaces. {\it Comment. Math. Helv.}
{\bf 60} (1985), 558--581.

%\bibitem{struwe1988}
%Struwe, M.: On the evolution of harmonic maps in higher dimensions.
%{\it J. Differential Geom.} {\bf 3} (1988),
%485--502.

\bibitem{tsai2018} Tsai, T. P.:  Lectures on Navier-Stokes Equations. Graduate Studies in Mathematics 192, American Mathematical Society.

%\bibitem{topping2002}
%Topping, P.: Reverse bubbling and nonuniqueness in the harmonic map %flow. {\it Int. Math. Research Notes} {\bf 10}
%(2002), 505--520.

%\bibitem{walkington2011}
%Walkington, N.J.: Numerical approximation of nematic liquid crystal %flows governed by the Ericksen-Leslie equations.
%{\it ESAIM: Math. Model. Numer. Anal.} {\bf 45}(3) (2011), 523--540.

%\bibitem{wang1996} 
%Wang, C. Y.: Bubble phenomena of certain Palais-Smale sequences from %surfaces to general targets. {\it Houston J. Math.} {\bf 22}(3) %(1996), 559---590.

\bibitem{wang2011} Wang, C. Y.: Well-posedness for the heat flow of harmonic maps and the liquid crystal flow with rough initial data
{\it Arch. Ration. Mech. Anal.} {\bf 200}(1) (2011), 1--19.

%\bibitem{wu2006}
%Wu, Z., Yin, J., Wang, C.: {\it Elliptic And Parabolic Equations}, %World Scientific Publishing Company, 2006.




%
%\bibitem{alouges}
%Alouges, F., Soyeur, A.: On global weak solutions for Landau-Lifshitz equations: existence and nonuniqueness. {\it Nonlinear Anal.} {\bf 18} (1992), 1071--1084.



}

\end{thebibliography}
\end{document}